\documentclass[11pt,draft, reqno]{amsart}

\usepackage[margin=18mm]{geometry}
\usepackage{amssymb, amsmath, amsthm, amscd, color}

\usepackage[T1]{fontenc}
\usepackage{eucal,mathrsfs,dsfont}
\usepackage{color}

\renewcommand{\leq}{\leqslant}
\renewcommand{\geq}{\geqslant}
\renewcommand{\le}{\leqslant}
\renewcommand{\ge}{\geqslant}

\def\eps{\varepsilon}

\definecolor{mno}{rgb}{0.5,0.1,0.5}

\newcommand{\R}{\mathds R}

\newcommand{\e}{\varepsilon}
\newcommand{\T}{\mathds T}

\newcommand{\Ee}{\mathds E}

\newcommand{\I}{\mathds 1}
\newcommand{\w}{\omega}
\newcommand{\N}{\mathds{N}}
\newcommand{\Z}{\mathds Z}
\newcommand{\sL}{\mathcal{L}}
\newcommand{\sE}{\mathcal{E}}
\newcommand{\sF}{\mathcal{F}}

\newcommand{\D}{\mathscr{D}}

\newcommand{\LL}{\mathcal{L}}

\newtheorem{theorem}{Theorem}[section]
\newtheorem{lemma}[theorem]{Lemma}

\theoremstyle{definition}

\newtheorem{remark}[theorem]{Remark}

\numberwithin{equation}{section}

\def\wt{\widetilde}

\begin{document}
\allowdisplaybreaks

\title[Quantitative periodic homogenization for
symmetric non-local stable-like operators] {Quantitative periodic homogenization
for
symmetric non-local stable-like operators}

\author{Xin Chen,\quad Zhen-Qing Chen,\quad Takashi Kumagai  \quad \hbox{and}\quad Jian Wang}

\date{}

\maketitle

\begin{abstract}
Homogenization
for non-local operators in periodic environments has been studied intensively.
So far, these works
are mainly devoted to the  qualitative   results, that is, to determine
explicitly the operators in the limit.
To the best of
authors'
 knowledge, there is no result concerning the
 convergence rates of the
homogenization for stable-like
operators in periodic environments.  In this paper, we establish a
quantitative  homogenization result
for symmetric $\alpha$-stable-like
operators on $\R^d$ with periodic coefficients. In particular,
we show that the
convergence rate for the solutions of associated Dirichlet problems
on a
bounded domain $D$ is of order
$$
\varepsilon^{(2-\alpha)/2}\I_{\{\alpha\in (1,2)\}}+\varepsilon^{\alpha/2}\I_{\{\alpha\in (0,1)\}}+\varepsilon^{1/2}|\log \e|^2\I_{\{\alpha=1\}},
$$
while, when the solution to
the equation in the limit is in
$C^2_c(D)$, the convergence rate becomes
 $$
\varepsilon^{2-\alpha}\I_{\{\alpha\in (1,2)\}}+\varepsilon^{\alpha}\I_{\{\alpha\in (0,1)\}}+\varepsilon |\log \e|^2\I_{\{\alpha=1\}}.
$$
This indicates that
the boundary decay behaviors  of the solution  to the equation in the limit
affects the convergence rate in the homogenization.

\medskip


\medskip

\end{abstract}
\allowdisplaybreaks

\section{Introduction}\label{section1}
This paper is devoted to quantitative results for homogenization of symmetric (non-local) stable-like
 operators in deterministic periodic media.
In line
with most of literature (e.g., \cite{CCKW1,HDS0, HDS, San, Sch}),
we restrict ourselves
to the Euclidean space setting.
Consider the following
 symmetric  stable-like operator $\sL$ on
$\R^d$:
\begin{align*}
\sL f(x)&:={\rm p.v.}\int_{\R^d}(f(y)-f(x))\frac{K(x,y)}{|x-y|^{d+\alpha}}\,dy \\
&:=\lim_{\eps \to 0} \int_{\R^d} {\bf1}_{\{|y-x|>\eps\}}(f(y)-f(x))\frac{K(x,y)}{|x-y|^{d+\alpha}}\,dy
\end{align*}
for $f\in C_c^2(\R^d)$.
Here
$\alpha\in (0,2)$, and
$K\in C^2(\R^d \times \R^d)$ is
  multivariate 1-periodic on $\R^d\times \R^d$ (that is, it can be viewed as a bounded
 $C^2$-smooth  function  defined on $\T^d\times \T^d$, where $\T^d:=\R^d/\Z^d$ is the $d$-dimensional torus)
so that
 \begin{equation}\label{e:1.1}
 K(x,y)=K(y,x) \quad \hbox{and} \quad
\Lambda^{-1}\le K(x,y)\le \Lambda
 \quad \hbox{for   }  x, y  \in \R^d
 \end{equation}
 with some
constant $\Lambda\ge 1$.
  The symmetric non-local operator $\sL$ generates a symmetric bilinear form for $u, v\in C^2_c(\R^d)$:
 \begin{equation} \label{e:1.2}
 \sE (u, v)=\frac12 \iint_{\R^d\times \R^d} (u(x)-u(y))(v(x)-v(y)) \frac{K(x,y)}{|x-y|^{d+\alpha}}\,dx \,dy.
 \end{equation}
 Denote by $\sF$ the completion of $C_c^2(\R^d)$ under the norm
 \begin{equation}\label{e:1.3}
 \| u\|_{\alpha/2, 2}:=  \left(\iint_{\R^d\times \R^d}
\frac{(u(x)-u(y))^2}{|x-y|^{d+\alpha}}
 \,dx dy +\int_{\R^d}u(x)^2 dx\right)^{1/2} .
  \end{equation}
 It is easy to check that
 \begin{equation}\label{e:1.4}
 \sF=W^{\alpha/2, 2} (\R^d) = \left\{ u\in L^2(\R^d; dx): \iint_{\R^d\times \R^d}
\frac{(u(x)-u(y))^2}{|x-y|^{d+\alpha}}
 \,dx dy <\infty \right\}
\end{equation}
and $(\sE, \sF)$ is a regular Dirichlet form on $L^2(\R^d; dx)$.
 Denote by $(\tilde  \sL, {\rm Dom}(\tilde  \sL))$
the $L^2$-generator
of $(\sE, \sF)$. Then $C^2_c(\R^d)\subset {\rm Dom}(\tilde \sL)$ and
$   \tilde  \sL f =\sL f $
  for $f\in C_c^2(\R^d)$.
It is known from \cite{CK03, CK08} that there is a symmetric  conservative irreducible
Feller process $X:=\{X_t\}_{t\ge 0}$ on $\R^d$ associated with the regular Dirichlet form $(\sE, \sF)$ on $L^2(\R^d; dx)$, which has the strong Feller property and can start from any $x\in \R^d$.
In the literature, $X$ is called a symmetric $\alpha$-stable-like process.
Since the coefficient $K(x,y)$ is multivariate 1-periodic on $\R^d\times \R^d$, one can regard $X$ as
a symmetric
$\T^d$-valued process.
It is exponentially ergodic and has the normalized Lebesgue measure on $\T^d$ as its  unique invariant probability measure.

For any $\e>0$, set
$$
X^\e:=\{X^\e_t\}_{t\ge 0}:= \{ \e X_{\e^{-\alpha}t} \}_{t\ge 0}.
$$
Clearly, $X^\e$ is a symmetric
$\R^d$-valued Feller  process having
the strong Feller property with associated Dirichlet form $(\sE^\e, W^{\alpha/2, 2}(\R^d))$ on $L^2(\R^d; dx)$,
where
\begin{equation} \label{e:1.5}
 \sE^\e (u, v)=\frac12 \iint_{\R^d\times \R^d} (u(x)-u(y))(v(x)-v(y)) \frac{K({x}/{\e},{y}/{\e})}{|x-y|^{d+\alpha}}\,dx dy.
 \end{equation}
 Denote by $(\sL_\e, {\rm Dom}  (\sL_\e))$ the $L^2$-generator of $(\sE^\e, W^{\alpha/2, 2}(\R^d))$.
Then $C_c^2(\R^d) \subset {\rm Dom}  (\sL_\e)$  and
 \begin{equation}\label{e:1.6}
\sL_\e f(x):={\rm p.v.}\int_{\R^d}(f(y)-f(x))\frac{K({x}/{\e},{y}/{\e})}{|x-y|^{d+\alpha}}\,dy,\quad f\in C_c^2(\R^d).
\end{equation}
Define
\begin{equation} \label{e:1.7a}
\bar K:= \iint_{\T^d\times \T^d}K(x,y)\,dx\,dy.
\end{equation}
Let $\bar X:=\{\bar X_t\}_{t\ge0}$ be a rotationally symmetric $\alpha$-stable L\'evy process on $\R^d$ with L\'evy measure
$\bar K |z|^{-(d+\alpha)} dz$.
Denote the infinitesimal generator of $\bar X$
 by $(\bar \sL , {\rm Dom}  (\bar \sL ))$.
Note that  $C_c^2(\R^d) \subset {\rm Dom}  (\bar \sL )$  and
 \begin{equation}\label{e:1.7}
\bar \sL f(x):={\rm p.v.}\int_{\R^d} (f(y)-f(x) )\frac{\bar K}{|y-x|^{d+\alpha}}\,dy  \quad \hbox{for } f\in C_c^2(\R^d).
\end{equation}
One can deduce from \cite[Theorem 4.1]{BKK}
that $X^\e$ converges weakly to $\bar X$  as $\eps \to 0$.
Alternatively, this weak convergence can also be established from
 the uniform heat kernel estimates in \cite[Theorem 1.1]{CK03} for symmetric $\alpha$-stable-like processes $\{X^\eps: \eps \in (0, 1]\}$
 and the $L^2$-convergence of the resolvents  of $X^\eps$ in \cite[Theorem 1.1]{KPZ} as $\eps \to 0$.
We  thus call $\bar \sL$ the homogenized operator of $\sL_\e$.
  The reader may refer to \cite{CCKW2,CCKW1} and the references therein for more details on homogenization of symmetric stable-like processes in stationary ergodic media and periodic homogenization problem of non-symmetric L\'evy-type processes.

Throughout the paper,
unless otherwise specified,  $D$ is a bounded
$C^{1,1}$
domain  in $\R^d$.
   For any $x\in D$, let $\delta_D(x):={\rm dist}(x,\partial D)$ be the distance from
$x$ to the boundary $\partial D$. For any $\theta\in \R$, $k\in \N$ and $0<\gamma\notin \N$, define
\begin{align*}
&[f]^{(\theta)}_{k;D}:=\sup_{x\in D}\left( \delta_D(x)^{k+\theta}|\nabla^k f(x)|\right),  \quad
[f]^{(\theta)}_{\gamma;D}:=\sup_{x,y\in D}\left(\min\left(\delta_D(x),\delta_D(y)\right)^{\gamma+\theta}\frac{|\nabla^{[\gamma]}f(x)-\nabla^{[\gamma]}f(y)|}{|x-y|^{\gamma-[\gamma]}}\right),\\
& C^\gamma_{(\theta)}(D):=\left\{f\in
C(\R^d):
\|f\|_{\gamma;D}^{(\theta)}:=[f]^{(\theta)}_{0;D}+[f]^{(\theta)}_{\gamma;D}<\infty \hbox{ and } f(x)=0 \hbox{ for } x\in D^c
\right\}.
\end{align*}
Here $[\gamma]$ denotes the largest integer not exceeding $\gamma$.
Let $\LL_\e$ and $\bar \LL$ be the operator defined by \eqref{e:1.6} and \eqref{e:1.7}, respectively.
Given  $h\in C_{(\alpha/2)}^{\beta}(D)\cap C(\bar D)$ with
 $\alpha+\beta >4$ so that neither $\beta$ nor $\alpha+\beta$ is an integer,
consider the following Dirichlet problem
\begin{equation}\label{e1-3}
\begin{cases}
\LL_\e u_\e(x)=h(x),&\ \ x\in D,\\
u_\e(x)=0,&\ x\in D^c.
\end{cases}
\end{equation}
Similarly,
consider the following
Dirichlet problem for the
homogenized
 operator $\bar\sL $:
\begin{equation}\label{e1-4}
\begin{cases}
\bar \LL \bar u(x)=h(x),&\ \ x\in D,\\
\bar u(x)=0,&\ x\in D^c.
\end{cases}
\end{equation}
 According to
 \cite[Proposition 1.4 and Theorem 1.5]{RS},
  the equation \eqref{e1-4} has the unique solution $\bar u\in C_{(-\alpha/2)}^{\alpha+\beta}(D)\cap C(\bar D)$.
On the other hand, according to Lemma \ref{l2-1} below, since $D$ is bounded and $h\in C(\bar D)$, the equation \eqref{e1-3} has the unique bounded solution $u_\e\in {\rm Dom}  (\LL_\e^D)\cap L^1(D;dx)$, where ${\rm Dom}  (\LL_\e^D)$ is the domain of the operator $\LL_\e$ with  zero Dirichlet condition on $D^c$.
Moreover, when the domain $D$ is more regular, higher order regularity on $u_\e$ will be obtained.
The reader is
referred to \cite[Theorem 3.1]{ABC} and \cite[Theorem 1.3]{ZZ} for more details on
regularity estimates for solutions to the Dirichlet problem of stable-like operators.
 The following is the main result of the paper, which establishes
an
explicit convergence rate of the solution $u_\e$ of \eqref{e1-3} in $L^1(D;dx)$ as $\e \to 0$.

\begin{theorem}\label{t2-1}
Suppose that $D$ is a bounded $C^{1,1}$
domain  in $\R^d$.
Let $u_\e$ and $\bar u$ be the unique
 solutions of \eqref{e1-3} and \eqref{e1-4} respectively for some
$h\in C_{(\alpha/2)}^{\beta}(D)$
with $\alpha+\beta > 4$, where neither $\beta$ nor $\alpha+\beta$ is an integer.
Then there exists a constant
$C_1>0$ so that for all $\e\in (0,1)$,
\begin{equation}\label{t2-1-1}
\|u_\e-\bar u\|_{L^1(D;dx)}\le
C_1\begin{cases} \e^{\alpha/2},&\quad \alpha\in (0,1),\\
\e^{1/2}(1+|\log \e|^2),&\quad \alpha=1,\\
\e^{(2-\alpha)/2},&\quad \alpha\in (1,2).
\end{cases}
\end{equation}
\end{theorem}

 As seen from the proof of Theorem \ref{t2-1} below, much more efforts
are
devoted to handling
blow up behaviours for $\nabla \bar u$ and $\nabla^2 \bar u$ near the boundary $\partial D$, as well as their
 interactions with the non-local operator $\LL_\e$.  Indeed, if the solution $\bar u$ of
 the   homogenized equation \eqref{e1-4}  is $C^2$-smooth with compact support in $D$
 (which roughly indicates that
we do not
 need to take care of the boundary behaviours of $\bar u$), then the
convergence rates of $u_\e$  to $\bar u$
will be better than those in \eqref{t2-1-1}.

\medskip

Now we consider \eqref{e1-3} and \eqref{e1-4} for $h\in C_b^1(\bar D)$ that does not need to vanish on $\partial D$ but
under the assumption that $\bar u\in C^2_c(D)$.

\begin{theorem}\label{t2-2}
Assume
that $D$ is a bounded Lipschitz domain in $\R^d$.
Suppose that $u_\e$ and $\bar u$ are the unique solutions of \eqref{e1-3} and \eqref{e1-4}, respectively,  for some
$h\in C_b^1(\bar D)$. If
$\bar u\in C_c^2(D)$, then there exists a constant
$C_2>0$ such that for all $\e\in (0,1)$,
\begin{equation}\label{t2-2-1}
\|u_\e-\bar u\|_{L^2(D;dx)}
\le C_2\begin{cases} \e^{\alpha},&\quad \alpha\in (0,1),\\
\e(1+|\log \e|^2),&\quad \alpha=1,\\
\e^{2-\alpha},&\quad \alpha\in (1,2).
\end{cases}
\end{equation}
\end{theorem}

The homogenization problem for differential operators with periodic coefficients has been investigated extensively in the literature.
The reader is referred to the monographs \cite{Ba,Be,JKO,Sh}.
Obtaining convergence rate is an important problem in the homogenization theory.
The first results in this direction for boundary value problems of elliptic operators with periodic coefficients in $L^2$-estimates are due to \cite{Be} and \cite{Ke1,Ke2}.
Since then, there are a lot of developments in this topic,  see
 \cite{Sh} and the references therein for more details.
In particular, we want to emphasize that the following sharp estimate
for elliptic differential operators
(which is comparable with the classical Berry-Esseen bounds for i.i.d.
random variables)
\begin{equation}\label{e1-1}
\|u_\e-\bar u\|_{L^2(D;dx)}\le
C_3\e
\end{equation}
has been established in \cite{Be,GS,KLS,KLS1,KLS2,ShZ,Su1} under different boundary conditions (including the Dirichlet boundary
condition and the Neumann boundary condition) and regularity
conditions. Here $u_\e$ (resp. $\bar u$) in \eqref{e1-1}
is the unique solution to
 \eqref{e1-3} (resp. \eqref{e1-4}) with $\sL_\e$ (resp.\ $\bar \sL$) being a second-order differential elliptic operator with rapidly oscillating and periodic coefficients (resp. constant coefficients). See also \cite{Bi,Bi2, Gr1, Gr2,Su,Zhi1,Zhi2} for related results.
We note that  there is some remarkable progress over the past decade
on quantitative stochastic homogenization of elliptic and parabolic equations
(for differential operators);
 see for instance \cite{AKM1, AKM2, GNO, GO}.

In recent years, the homogenization for non-local operators in periodic environments has also been studied intensively, see
\cite{CCKW1,HDS0,HDS,KPZ,PZ,San,Sch,SU,T}. In all these papers only
the qualitative homogenization results  for non-local operators   are   investigated;
 that is, these papers determine the limit operators and establish the convergence of the processes or the solutions of the non-local partial differential equations  to those of  the limit operators.
The only existing literature we know
on quantitative  homogenization of non-local operators is
\cite{PSSZ}, where
the $L^2$-convergence
rate is obtained by a spectral method for 1-resolvents of a class of non-local operators of convolution type having $L^2$ integrable jumping kernels.
We note
that the limit operators in \cite{PSSZ} are
second order  elliptic differential operators,
rather than
non-local operators,  and that the
rate of convergence   for the associated Dirichlet problem is not investigated in \cite{PSSZ}.
 \emph{To the best of our knowledge, there is no result concerning convergence rates of the
homogenization for stable-like
operators in periodic environments}.  In this paper we will address this problem for symmetric $\alpha$-stable-like operators.
It is interesting to observe the difference between our results \eqref{t2-1-1} and \eqref{t2-2-1} for non-local operators
and \eqref{e1-1} for differential operators
as well as to the non-local operator investigated in \cite{PSSZ} where the jumping kernel has finite second moment.
We also want to mention that the interior convergence rate for symmetric $\alpha$-stable-like operators in
i.i.d. environments,  where the boundary behavior of the solution was not taken into account, has been obtained in the paper \cite{CCKW-new1}.

Below are some additional comments on the results of this paper.

\begin{itemize}
\item [(1)] The interior convergence rate \eqref{t2-2-1} seems to be optimal, since it is
in line with
the Berry-Esseen estimates for i.i.d.\ random variables having stable laws as limit, see \cite{Ban,CNX,X} for example.
We also believe that the
convergence rates in \eqref{t2-1-1} for the associated Dirichlet problem are near optimal for the methods used in this paper,
 due to the blow up behaviors of $\nabla \bar u$ and $\nabla^2 \bar u$ near the boundary $\partial D$.
These blow ups
make the convergence rate slower than  \eqref{t2-2-1}
which can be viewed
as the interior convergence rate, even if the domain $D$ is smooth. This is different
from the case for the Dirichlet problem of differential operators
as limit equations,
where
$\bar u$ are known to decay at the boundary in a linear rate and so $\nabla u$  and $\nabla^2 u$ are all bounded,
see e.g. \cite{GS,KLS,PSSZ,Sh}
for more details.

\item [(2)]
We only
give the $L^1$-convergence rate in \eqref{t2-1-1}.
From \eqref{t2-1-1} and Lemma \ref{l2-1} below,
 we could deduce the $L^2$-convergence rate from the inequality
$$\| u_\eps -\bar u\|_{L^2(D; dx)} \leq C_4 \| h\|_\infty^{1/2}\| u_\eps -\bar u\|_{L^1(D; dx)}^{1/2},
$$
but we do not think
the resulting
 $L^2$-convergence rate is optimal.
As illustrated in
\cite{ABC,RS,ZZ}, it seems difficult to verify the $H^2$-estimate $\int_D |\nabla^2 \bar u(x)|^2 \,dx<\infty$ for the
fractional Laplacian operator, which is a key point to establish the $L^2$-convergence for differential operators
in \cite{GS,KLS}. Because of the lack of such kind
of the $H^2$-estimate,
it does not seem easy to apply
the integration by parts formula to establish
a sharper $L^2$-convergence rate.
 On the other hand, ignoring the boundary behaviors of $\bar u$, we can obtain the $L^2$-interior convergence rate \eqref{t2-2-1}.
Indeed, in the
proof of \eqref{t2-2-1} we can use the integration by parts formula to reduce the requirements into the
required regularity
of $\bar u$ up to $C^2$, while in our proof of \eqref{t2-1-1} higher order regularities  on $\bar u$ are needed.

\item [(3)] Compared with the approach for the quantitative homogenization of differential operators,
there are two essential different ingredients in the proof of Theorem \ref{t2-1}.
One is the quantitative estimates
 \eqref{l2-2-1} and \eqref{l2-3-1} below
   for the average
 in $z$-variable
	in the homogenization problem of $\alpha$-stable-like operators, which represents
the jumping size of the associated stable-like process. However, for differential operators only the averaging properties of
the position variable $x$ are needed.
The other
is to give detailed estimates on the interactions between blow up behaviors of $\nabla \bar u$ and
$\nabla^2 \bar u$ near $\partial D$ and the effect of the non-local operator $\LL_\e$, see Lemmas \ref{l2-5}--\ref{l2-7} below for details. In particular, in the arguments of
two
ingredients
above,
it seems natural to use the $C_{(-\alpha/2)}^{\alpha+\beta}$-norm (with suitable choice of $\beta>0$) instead of the
$H^2$-norm for $\bar u$, and the constant
$C_1$
in \eqref{t2-1-1} only depends on
$\sup_{0\le k\le 4}\sup_{x\in D}\left( \delta_D(x)^{k-\alpha/2}|\nabla^k \bar u(x)|\right)$.
\end{itemize}

\ \

The rest of this paper is arranged as follows.
We present in Section \ref{S:2} some preliminaries that are needed later in this paper.
They include  some regularity results for  solutions to the Poisson equation associated with symmetric stable-like operators, and several averaging estimates for $\sL_\e f$.  Their proofs involve some delicate and lengthy estimates for distance functions and cut-off functions.
 The proofs of Theorems \ref{t2-1} and \ref{t2-2} are given in Section \ref{section3}.

\section{Preliminaries}\label{S:2}

In this section we give some preliminaries that are needed for the proofs of our main theorems.

\subsection{Solution to the Poisson equation}
Since
$K\in C^2(\R^d \times \R^d)$ is multivariate 1-periodic on $\R^d\times \R^d$, we can view
the corresponding symmetric Feller process $X:=\{X_t\}_{t\ge 0}$ on $\R^d$  and its $L^2$-generator
$(\sL, {\rm Dom}(\sL)) $ as a  symmetric Feller  process
$X ^{\T^d}:= \{X_t^{\T^d}\}_{t\ge 0}$ on  the torus $\T^d$ and  a generator $(\LL,
{\rm Dom}  (\sL^{\T^d}))$  on $\T^d$,
respectively.

\begin{lemma}\label{l2-4}
For every $f\in C(\T^d)$ with $\int_{\T^d}f(x)\,dx=0$, there exists
a unique  $\psi_f\in {\rm Dom}  (\sL^{\T^d})\cap C(\T^d)$ such that
\begin{equation}\label{l2-4-1}
\LL \psi_f=-f
\quad \hbox{on } \T^d
\end{equation}
and $\int_{\T^d}\psi_f(x)\,dx=0$.
Moreover, the following properties hold.
\begin{itemize}
\item [(1)] For any $\alpha\in (0,2)$,   there is a constant $C_1>0$ so that
\begin{equation}\label{l2-4-3}
\|\psi_f\|_\infty\le C_1\|f\|_\infty.
\end{equation}

\item [(2)] If $\alpha\in (1,2)$, then $\psi_f\in C^1(\T^d)$ and there is a constant $C_2>0$ so that
\begin{equation}\label{l2-4-2}
\|\psi_f\|_\infty+\|\nabla \psi_f\|_\infty\le C_2 \|f\|_\infty.
\end{equation}

\item [(3)] If $\alpha=1$, then there exist constants $\theta\in (0,1)$ and $C_3>0$ such that
\begin{equation}\label{l2-4-4}
\sup_{x,y\in \T^d}\frac{|\psi_f(x)-\psi_f(y)|}{|x-y|^\theta}\le C_3\|f\|_\infty.
\end{equation}
\end{itemize}
\end{lemma}

\begin{proof}
  Denote by $\{P_t\}_{t\ge 0}$ the transition semigroup of the Feller process $X^{\T^d}$ on $\T^d$.
 By \cite[Theorem 1.1]{CK03}, it has a strictly positive, jointly continuous transition density function
$p(t, x, y)$ with
\begin{equation}\label{eL2.5}
c_0^{-1} \left(t^{-d/\alpha} \wedge \frac{t}{|x-y|^{d+\alpha}} \right) \leq p(t, x, y)
\leq c_0  \left(t^{-d/\alpha} \wedge \frac{t}{|x-y|^{d+\alpha}} \right)
\quad \hbox{on } (0, 1]\times \T^d \times T^d
\end{equation}
for some $c_0>1$.
 Since $\T^d$ is compact,  $X^{\T^d}$ has the normalized Lebesgue measure on $\T^d$ as its equilibrium probability distribution as well as a stationary measure.
 It is standard to verify that $\{P_t\}_{t\ge 0}$ is uniformly exponentially ergodic
(see e.g. \cite{DMT}),  that is,
\begin{equation}\label{l2-4-5}
\|P_t g\|_\infty \le c_1e^{-c_2 t}\|g\|_\infty  \quad \hbox{for any } g\in C(\T^d)\ {\rm with}\ \int_{\T^d}g(x)\,dx=0.
\end{equation}
Since for each $t>0$, the operator $P_t$ is Hilbert-Schmidt, it is compact in $L^2(\T^d; dx)$ and thus has discrete $L^2$-spectrum $\{e^{-\lambda_k t}: k=1, 2, \cdots\}$
with $0=\lambda_0 < \lambda_1\leq \lambda_2 \leq \cdots  $ independent of $t>0$.   Denote by $\varphi_k$ the corresponding eigenfunctions so that $\{\varphi_k: k\geq 0\}$
forms an orthonormal basis of $L^2(\T^d; dx)$ with $\varphi_0 = 1/|\T^d|$.
Then, for each $t>0$,
\begin{equation}\label{e:2.7}
\|P_t g\|_2 \le   e^{-\lambda_1 t}\|g\|_2  \quad \hbox{for every }  g\in L^2 (\T^d; dx)\ {\rm with}\ \int_{\T^d}g(x)\,dx=0.
\end{equation}
Denote by $\langle\cdot,\cdot\rangle_{L^2(\T^d;dx)}$
the
 inner product on $L^2(\T^d;dx)$.
For $f\in
C(\T^d)$ with $
\int_{\T^d}f(x)\,dx=0$, define
 $$
 \psi_f := \int_0^\infty P_t  f dt  = \sum_{k=1}^\infty \lambda_k^{-1} \langle f, \varphi_k\rangle_{L^2(\T^d;dx)} \varphi_k,
 $$
  which converges both in $L^2$ and uniformly on $\T^d$
in view of  \eqref{l2-4-5} and \eqref{e:2.7}.
Clearly, $\int_{\T^d} \psi_f (x) dx=0$.
 Moreover,
it follows from  the spectral theory that $\psi_f \in
{\rm Dom}  (\sL^{\T^d}) $ with
$$
\LL \psi_f = \sum_{k=1}^\infty (-\lambda_k) \langle\psi_f, \varphi_k\rangle_{L^2(\T^d;dx)} \varphi_k = - \sum_{k=1}^\infty
\langle f, \varphi_k\rangle_{L^2(\T^d;dx)} \varphi_k = -f.
$$
So $\psi_f$ is a solution to \eqref{l2-4-1}
with $\int_{\T^d}\psi_f(x)\,dx=0$. Such a solution is unique as any $u\in {\rm Dom}  (\sL^{\T^d})$ having
$ \LL u=0$ is constant on $\T^d$.
It follows from  \eqref{l2-4-5} that
\eqref{l2-4-3} holds for every $\alpha\in (0,2)$.

\medskip

(i) When $\alpha\in (1,2)$, for $f\in C_c^2(\R^d)$,
$$
\sL f(x) = \int_{\R^d} ( f(x+z) -f (x) -z \cdot \nabla f(x)) \frac{K(x, x+z)}{|z|^{d+\alpha}} dz + b(x) \cdot \nabla f(x),
$$
where
$$
b(x)=\frac12  \int_{\R^d} \frac{z (K(x, x+z)-K(x, x-z))}{|z|^{d+\alpha}} \,dz\in C_b(\R^d).
$$
In particular, the drift term $b(x)$ belongs to Kato's class in the sense of \cite[Definition 1.3]{CZ1}; see \cite[Remark 1.4]{CZ1}.
If $\alpha\in (1,2)$, then, by
\cite[Theorem 1.5(vi)]{CZ1},
\begin{align*}
\|\nabla P_t g\|_\infty \le c_3t^{-1/\alpha}\|g\|_\infty,\quad t\in (0,1],\ g\in C(\T^d).
\end{align*}
Thus,
\begin{align*}
\|\nabla \psi_f\|_\infty&\le \int_0^1 \|\nabla P_t f\|_\infty \,dt+\int_1^\infty \|\nabla P_1\left(P_{t-1} f\right)\|_\infty \,dt\\
&\le c_3
\left(\int_0^1 t^{-1/\alpha}\|f\|_\infty \,dt +\int_1^\infty \|P_{t-1}f\|_\infty \,dt\right)\\
&\le c_4
\left(\int_0^1 t^{-1/\alpha}\|f\|_\infty\, dt +\int_1^\infty e^{-c_5 (t-1)}\|f\|_\infty \,dt\right)\le c_6\|f\|_\infty.
\end{align*}
So \eqref{l2-4-2} is proved.

(ii) When $\alpha=1$, by \cite[Theorem 4.14]{CK03}, there exist constants $c_7>0$, $\kappa>1$ and $\beta\in (0,1)$ such that
\begin{equation}\label{l2-4-6}
|P_tg(x)-P_tg(y)|\le c_7\|g\|_\infty t^{-\kappa}|x-y|^{\beta},\quad t\in (0,1],\, x,y\in \T^d \hbox{ and } g\in C(\T^d).
\end{equation}
Write
\begin{align*}
|\psi_f(x)-\psi_f(y)|&\le \left(\int_0^{|x-y|^{{\beta}/{\kappa}}}+
\int_{|x-y|^{{\beta}/{\kappa}}}^1+\int_1^\infty\right)|P_t f(x)-P_t f(y)|\,dt\\
&=:I_1(x,y)+I_2(x,y)+I_3(x,y).
\end{align*}
It holds that
\begin{align*}
I_1(x,y)&\le 2\|f\|_\infty |x-y|^{{\beta}/{\kappa}}.
\end{align*}
Applying \eqref{l2-4-6}, we obtain
\begin{align*}
I_2(x,y)&\le c_8\|f\|_\infty |x-y|^{{\beta}(-\kappa+1)/{\kappa}}|x-y|^{\beta}=c_8\|f\|_\infty |x-y|^{\beta/\kappa}.
\end{align*}
Using \eqref{l2-4-5} and \eqref{l2-4-6}, we deduce
\begin{align*}
I_3(x,y)&=\int_1^\infty \left|P_1(P_{t-1}f)(x)-P_1(P_{t-1}f)(y)\right|\,dt\\
&\le c_7|x-y|^{\beta}\int_1^\infty \|P_{t-1}f\|_\infty\, dt\\
&\le c_9|x-y|^{\beta}\int_1^\infty e^{-c_{10}(t-1)}\|f\|_\infty dt\le c_{11}\|f\|_\infty |x-y|^\beta.
\end{align*}

Combining
all the estimates for $I_1(x,y)$, $I_2(x,y)$ and $I_3(x,y)$ yields the desired conclusion
\eqref{l2-4-4} with $\theta=\beta/\kappa$ when $\alpha=1$.
\end{proof}

  Recall that  $\{X_t^\e\}_{t\ge 0}:= \{\e X_{\e^{-\alpha}t}\}_{t\ge 0}$,
  the associated
   Dirichlet form on $L^2(\R^d; dx)$ is
$(\sE^\e, W^{\alpha/2, 2}(\R^d))$ defined by \eqref{e:1.5}, and that
  $(\sL_\e, {\rm Dom}  (\sL_\e))$
  is
  the $L^2$-generator of   $(\sE^\e, W^{\alpha/2, 2}(\R^d))$ on $L^2(\R^d; dx)$.
  For an open subset $D\subset \R^d$, denote by    $\{X_t^{\e,D}\}_{t\ge 0}$ the subprocess of
$\{X_t^\e\}_{t\ge 0}$ killed upon leaving $D$.  Its corresponding Dirichlet form $(\sE^\e,   W^{\alpha/2, 2}_0(D))$ is given by \begin{align*}
 \sE^\e(f,g)&=\frac{1}{2}\iint_{\R^d\times\R^d} (f(x)-f(y)) (g(x)-g(y))\frac{K\left({x}/{\e},{y}/{\e}\right)}{|x-y|^{d+\alpha}}\,dx\,dy,
 \quad f, g \in   W^{\alpha/2, 2}_0(D)  ,   \\
  W^{\alpha/2, 2}_0(D)   &:=
 \overline{C_c^2(D)}^{ \|\cdot\|_{\alpha/2, 2}}.
\end{align*}
Denote the infinitesimal generator of $(\sE^\e,   W^{\alpha/2, 2}_0(D))$  by
$(\LL_\e,{\rm Dom}  (\LL_\e^D)).$

\begin{lemma}\label{l2-9}
 Suppose that $D$ is a bounded  open subset of $\R^d$.  Then there exists a constant $C_1>0$ such that for every  $\e\in (0,1)$,
\begin{equation}\label{l2-9-1}
\int_D |f(x)|^2 \,dx\le C_1\sE^\e(f,f) \quad \hbox{for } f\in
  W^{\alpha/2, 2}_0(D).
\end{equation}
\end{lemma}

\begin{proof}
Since the function $K (x, y)$ is bounded between two positive constants, it suffices to show that
there is a constant $c_1>0$ so that
\begin{equation}\label{l2-9-2}
\int_D |f(x)|^2 \,dx\le c_1\iint_{\R^d\times\R^d}\frac{|f(x)-f(y)|^2}{|x-y|^{d+\alpha}}\,dx\,dy  \quad
\hbox{for } f \in W^{\alpha/2, 2}_0(D) .
\end{equation}
Let $\{\bar P_t^D\}_{t\ge 0}$ be the Dirichlet semigroup on $D$ associated with
a rotationally symmetric
 $\alpha$-stable L\'evy process $\bar X$, whose infinitesimal generator is $\bar \LL$. Since $D$ is bounded, it well known that
$\bar P_t^D$ is compact on $L^2(D;dx)$ (see e.g. \cite{CKS} or \cite{CS1}), and
its first eigenvalue is strictly
less than 1.
This fact implies the inequality \eqref{l2-9-2} immediately, and so we can obtain the desired assertion.
\end{proof}

\begin{lemma}\label{l2-1}
Suppose that $D$ is a bounded  open subset of $\R^d$ and  $f\in L^2(D;dx)$.
Then for every $\e >0$, there is a unique $\varphi_\e \in
{\rm Dom}  (\LL_\e^D)$ so that
 $ \LL_\e \varphi_\e =  f$ on $D$, and
 there exists a constant
$C_1>0$
 such that for all $\e \in (0,1)$,
\begin{equation}\label{l2-1-1}
 \int_D|\varphi_\e(x)|\,dx\le
C_1\int_D |f(x)|\,dx.
\end{equation}
Moreover, if, in  addition,
 $f$ is bounded on $D$, then there is a constant
$C_2>0$
such that for all $\e \in (0,1)$,
$$\|\varphi_\e\|_{L^\infty(D;dx)}\le
C_2.$$
\end{lemma}

\begin{proof}
Since $D$ is bounded,
$D \subset B(0,R)$ for some $R>0$.
Suppose $f\in L^2(D; dx)$. By  \eqref{l2-9-1} and the Riesz representation theorem, there is a unique $\varphi_\e\in W^{\alpha/2, 2}_0(D)$
so that
$$
\sE^\e (\varphi_\e, g) =   - \int_D f(x) g(x) \,dx \quad \hbox{for every } g\in W^{\alpha/2, 2}_0(D).
$$
It follows that $\varphi_\e \in {\rm Dom}  (\LL_\e^D)$ with $
\LL_\e\varphi_\e=  f$.
In fact, by \cite{CF},
\begin{equation}\label{l2-1-2}
\varphi_\e(x)=\Ee_x\left[\int_0^{\tau_D^\e}f(X_s)\,ds\right]=\int_D G_D^\e(x,y)f(y)\,dy,\quad
{\rm a.e.}\ x\in D.
\end{equation}
where $\tau_D^\e:=\inf\{t>0: X_t^\e\notin D\}$ is the first exit time from $D$ for the process
$\{X_t^\e\}_{t\ge 0}$, and
$ G_D^\e  (x, y)$ is the Green function of $X^\e$ in $D$.
Note that
\begin{equation}\label{Add1}
 G^\e_D 1 (x):=  \int_D G_D^\e(x,y)\,dx=\Ee_y\left[\tau_D^\e\right],\quad  y\in D.
\end{equation}
  By the symmetry of $G_D^\e(x, y)$ in $(x, y)$, we have from \eqref{l2-1-2}   that
\begin{equation}\label{l2-1-3}
\int_D|\varphi_\e(x)|\,dx
 \leq \int_D |f(x)| G^\e_D 1(x) \,dx
\le \sup_{y\in D}\Ee_y\left[\tau_D^\e\right]\cdot \int_D|f(x)|\,dx.
\end{equation}
  By  \cite[Lemma 5.1]{CK03},  there is a constant  $c_1>0$ that depends only on
the bound $\Lambda\ge 1$
in \eqref{e:1.1} for the function
$K(x, y)$ and the dimension $d$ so that
\begin{equation}\label{Add2}
\sup_{\e\in (0,1)}\sup_{y\in D}\Ee_y\left[\tau_D^\e\right]\le
\sup_{\e\in (0,1)}\sup_{y\in B(0,R)}\Ee_y\left[\tau_{B(0,R)}^\e\right]\le
c_1  R^\alpha.
\end{equation}
This together with \eqref{l2-1-3} yields the estimate \eqref{l2-1-1}.

Furthermore, the last assertion is a direct consequence of \eqref{l2-1-2}, \eqref{Add1} and \eqref{Add2} as well as the fact that $ G_D^\e(x,y)= G_D^\e(y,x)$ for all $x,y\in D$.
\end{proof}

\subsection{Distance functions and cut-off functions}

It is well known that (see e.g. \cite[Theorem 2, p.\ 171]{Stein})
for any bounded open set $D\subset \R^d$, there is a $C^\infty$-smooth regularizing distance function $\tilde \delta_D$ on $D$ so that there are
positive constants $c_1$, $c_2$ and $c_3$ such that
$$
c_1 \delta_D (x) \leq \tilde \delta_D (x) \leq c_2 \delta_D(x)  \quad \hbox{and} \quad
  | \nabla \tilde \delta_D (x)| \leq c_3
\quad \hbox{for every } x\in D.
$$
For any $r>0$, we set $D_r:=\{x\in D:
\wt \delta_D(x)>r\}$.

\begin{lemma}\label{l2-5}
Suppose that $D$ is a bounded Lipschitz domain  in $\R^d$.
Then the following statements hold.
\begin{itemize}
\item [(1)] When $\alpha\in (1,2)$, there is a constant $C_1>0$ so that for every $s,\e\in (0,1)$ and $x\in D$,
\begin{equation}\label{l2-5-1}
\begin{split}
& \int_{\{|z|>s \e\}}\delta_D(x+z)^{-2+\alpha/2}\I_{\{x+z\in D_{3\e}\}}\frac{1}{|z|^{d+\alpha-1}}\,dz\\
&\le C_1\left((s\e)^{1-\alpha}\delta_D(x)^{-2+\alpha/2}+\e^{-1+\alpha/2}\delta_D(x)^{-\alpha}\right)\I_{D_{4\e}}(x)+C_1s^{1-\alpha}\e^{-1-\alpha/2}\I_{D\backslash D_{4\e}}(x)
\end{split}
\end{equation} and
\begin{equation}\label{l2-5-2}
\begin{split}
\int_{\{|z|>s \e\}}\I_{\{x+z\in D\backslash D_{3\e}\}}\frac{1}{|z|^{d+\alpha-1}}\,dz\le
C_1\e\delta_D(x)^{-\alpha}\I_{D_{4\e}}(x)+C_1(s\e)^{1-\alpha}\I_{D\backslash D_{4\e}}(x).
\end{split}
\end{equation}
\item [(2)] When $\alpha\in (0,1]$, there is a constant $C_2>0$ so that for every $s,\e\in (0,1)$ and $x\in D$,
\begin{equation}\label{l2-5-3}
\begin{split}
&  \int_{\{s\e<|z|\le s\}}\delta_D(x+z)^{-2+\alpha/2}\I_{\{x+z\in D_{3\e}\}}\frac{1}{|z|^{d+\alpha-1}}\,dz\\
&\le C_2\left( \delta_D(x)^{-2+\alpha/2}\left(
\delta_D(x)^{1-\alpha}\I_{\{\alpha\in (0,1)\}}+
(1+|\log \e|)\I_{\{\alpha=1\}}\right) +\e^{-1+\alpha/2}
\delta_D(x)^{-\alpha}\I_{\{s>\delta_D(x)/4\}}
\right) \I_{D_{4\e}}(x)\\
&\quad
+
C_2\e^{-2+\alpha/2}\left(s^{1-\alpha}\I_{\{\alpha\in (0,1)\}}+(1+|\log \e|)\I_{\{\alpha=1\}}\right)
\I_{D\backslash D_{4\e}}(x),
\end{split}
\end{equation}  and
\begin{equation}\label{l2-5-4}
\begin{split}
&  \int_{\{s\e<|z|\le s\}}\I_{\{x+z\in D\backslash D_{3\e}\}}\frac{1}{|z|^{d+\alpha-1}}\,dz\\
&\le  C_2\e\delta_D(x)^{-\alpha} \I_{D_{4\e}}(x)+
C_2\left(s^{1-\alpha}\I_{\{\alpha\in (0,1)\}}+(1+|\log \e|)\I_{\{\alpha=1\}}\right)
\I_{D\backslash D_{4\e}}(x),
\end{split}
\end{equation}
 as well as
\begin{equation}\label{l2-5-5}
\begin{split}
& \int_{\{|z|>s\}}\left(\delta_D(x+z)^{-1+\alpha/2}\I_{\{x+z\in D_{3\e}\}}+\e^{-1+\alpha/2}\I_{\{x+z\in D\backslash D_{3\e} \}}\right)\frac{1}{|z|^{d+\alpha}}\,dz
\le C_2\e^{-1+\alpha/2}s^{-\alpha}.
\end{split}
\end{equation}
\item [(3)] When $\alpha\in (0,1)$, for every $\beta\in (\alpha/2,\alpha)$ there is a constant $C_3>0$ so that for all $\e\in (0,1)$ and $x\in D$,
\begin{equation}\label{l2-5-3a}
\begin{split}
&  \int_{\{|z|>\e\}}\delta_D(x+z)^{-1-\beta+\alpha/2}\I_{\{x+z\in D_{3\e}\}}\frac{1}{|z|^{d+\alpha-\beta}}\,dz\\
&\le C_3\left(\e^{-(\beta-\alpha/2)}\delta_D(x)^{-1-(\alpha-\beta)}+\e^{-(\alpha-\beta)}\delta_D(x)^{-1-\beta+\alpha/2}\right)\I_{D_{4\e}}(x)+
C_3\e^{-1-\alpha/2}\I_{D\backslash D_{4\e}}(x)
\end{split}
\end{equation} and
\begin{equation}\label{l2-5-4a}
\begin{split}
&\quad \int_{\{|z|>\e\}}\I_{\{x+z\in D\backslash D_{3\e}\}}\frac{1}{|z|^{d+\alpha-\beta}}\,dz
\le  C_3\e\delta_D(x)^{-1-(\alpha-\beta)}\I_{D_{4\e}}(x)+
C_3\e^{-(\alpha-\beta)}\I_{D\backslash D_{4\e}}(x).
\end{split}
\end{equation}
\end{itemize}
\begin{proof}
(1) Suppose that $\alpha\in (1,2)$. Let $s,\e\in (0,1)$. Then, for every $x\in D_{4\e}$, it holds that
\begin{align*}
&\int_{\{|z|>s \e\}}\delta_D(x+z)^{-2+\alpha/2}\I_{\{x+z\in D_{3\e}\}}\frac{1}{|z|^{d+\alpha-1}}\,dz\\
&\le c_{1}\Bigg(\delta_D(x)^{-2+\alpha/2}\int_{\{s\e<|z|\le {\delta_D(x)}/{4}\}}\frac{1}{|z|^{d+\alpha-1}}\,dz
+\int_{\{
x+z\in D:
|z|>{\delta_D(x)}/{4},\delta_D(x+z)>3\e\}}
\frac{\delta_D(x+z)^{-2+\alpha/2}}{|z|^{d+\alpha-1}}\,dz\Bigg)\\
&=:I_{1}(x)+I_{2}(x),
\end{align*}
where in the inequality above we used
\begin{align}\label{l2-5-7}
\delta_D(x+z)\ge \delta_D(x)/2 \quad \hbox{for }
 |z|\le {\delta_D(x)}/{4}.
\end{align}
It is clear that
\begin{align*}
I_{1}(x)&\le c_{2}\delta_D(x)^{-2+\alpha/2}(s\e)^{1-\alpha}.
\end{align*}
On the other hand,
\begin{equation}\label{l2-2-6}
\begin{split}
I_{2}(x)&\le \sum_{k=0}^{\infty}
\int_{\{
x+z\in D:
{2^k\delta_D(x)}/{4}<|z|\le {2^{k+1}\delta_D(x)}/{4}, \delta_D(x+z)>3\e\}}
\frac{\delta_D(x+z)^{-2+\alpha/2}}{|z|^{d+\alpha-1}}\,dz\\
&\le c_{3}\sum_{k=0}^\infty(2^k\delta_D(x))^{-d-\alpha+1}\int_{\{
x+z\in D:
{2^k\delta_D(x)}/{4}<|z|\le {2^{k+1}\delta_D(x)}/{4},\delta_D(x+z)>3\e\}}
\delta_D(x+z)^{-2+\alpha/2}\,dz\\
&=:c_3\sum_{k=0}^\infty (2^k\delta_D(x))^{-d-\alpha+1} I_{2k}(x).
\end{split}
\end{equation}
By
a
change of variable, we have
\begin{align*}
I_{2,k}(x)=\int_{\{z\in D: 2^k\delta_D(x)/4<|z-x|\le 2^{k+1}\delta_D(x)/4, \delta_D(z)>3\e\}}\delta_D(z)^{-2+{\alpha}/{2}}\,dz.
\end{align*}
If $\delta_D(x)>c_4 2^{-k}$ for some $c_4>0$ to be determined later, then it holds that
\begin{align*}
I_{2,k}(x)&\le \int_{\{z\in D: \delta_D(x)>3\e\}}\delta_D(z)^{-2+{\alpha}/{2}}\,dz\\
&\le c_5\int_{3\e}^\infty s^{-2+{\alpha}/{2}}\,ds\le c_6\left(2^{k+1}\delta_D(x)\right)^{d-1}\e^{-1+{\alpha}/{2}},
\end{align*}
where the last inequality due to the fact $2^k\delta_D(x)\ge c_4$.

When $\delta_D(x)\le c_4 2^{-k}$, since $D$ is bounded and Lipschitz, we can choose $c_4>0$ small enough so that there exists
a change of coordinates $\Psi:B\left(x,2^{k+1}\delta_D(x)/4\right)\cap D\to  B\left(x',2^{k+1}\delta_D(x)\right)\cap \mathbb{H}_+$ with
$\mathbb{H}_+:=\{z\in \R^d: z=(z_1,\cdots, z_d),\ z_d>0\}$ which satisfies that
\begin{itemize}
\item [(i)] $\Psi$ is a homeomorphism, $\Psi(x)=x'$ and
\begin{align*}
&|\Psi(x_1)-\Psi(x_2)|\le c_7|x_1-x_2| \quad \hbox{for } x_1,x_2\in B\left(x,2^{k+1}\delta_D(x)/4\right)\cap D,\\
&|\Psi^{-1}(z_1)-\Psi^{-1}(z_2)|\le c_7|z_1-z_2| \quad  \hbox{for }   z_1,z_2\in  B\left(x',2^{k+1}\delta_D(x)\right)\cap \mathbb{H}_+;
\end{align*}
\item [(ii)]
$$c_8\delta_D(z)\le \delta_{\mathbb{H}_+}\left(\Psi(z)\right)\le c_7\delta_D(z)
\quad \hbox{for }  z\in B\left(x,2^{k+1}\delta_D(x)/4\right).$$
\end{itemize}
Here, $\delta_{\mathbb{H}_+}(z):=z_d$, for any $z\in \mathbb{H}_+$, denotes the distance from $z\in \mathbb{H}_+$ to the boundary $\mathbb{H}_+$. In particular,
we remark
that all the constants $c_i$ above are independent of $x$.

According to these properties above and
a
change of variable
$z\mapsto \Psi(z)$, we derive
\begin{align*}
I_{2,k}&\le c_9\int_{\{z\in \mathbb{H}_+: |z-x'|\le c_{10}2^{k+1}\delta_D(x), \delta_{\mathbb{H}_+}(z)=z_d>c_{11}\e\}}\delta_{\mathbb{H}_+}(z)^{-2+\alpha/2}\,dz\\
&\le c_9\int_{\{z\in \mathbb{H}_+: \sup_{1\le i \le d-1}|z_i-x'_i|\le c_{10}2^{k+1}\delta_D(x), z_d>c_{11}\e\}}z_d^{-2+\alpha/2}\,dz\\
&\le c_{12}\left(2^{k+1}\delta_D(x)\right)^{d-1}\int_{c_{11}\e}^\infty s^{-2+{\alpha}/{2}} \,ds\le c_{13}\left(2^{k+1}\delta_D(x)\right)^{d-1}\e^{-1+\alpha/2}.
\end{align*}
Therefore, summarising all the estimates above together, we arrive at
\begin{equation}\label{l2-5-6a}\begin{split}
I_2(x)&\le c_3\sum_{k=0}^\infty (2^k\delta_D(x))^{-d-\alpha+1} I_{2,k}(x)\le c_{14}\e^{-1+\alpha/2}\sum_{k=0}^\infty(2^k\delta_D(x))^{-\alpha}
\le c_{15}\e^{-1+\alpha/2}\delta_D(x)^{-\alpha}.
\end{split}\end{equation}

When $x\in D\backslash D_{4\e}$, for any $s,\e\in (0,1)$, it holds that
\begin{equation}\label{l2-5-6}
\begin{split}
& \int_{\{|z|>s \e\}}\left(\delta_D(x+z)^{-2+\alpha/2}\I_{\{x+z\in D_{3\e}\}}+\e^{-2+\alpha/2}\I_{\{x+z\in D \backslash D_{3\e}\}}\right)\frac{1}{|z|^{d+\alpha-1}}\,dz\\
&\le c_{16}\e^{-2+\alpha/2}\int_{\{|z|>s\e\}}\frac{1}{|z|^{d+\alpha-1}}\,dz\le c_{17}s^{1-\alpha}\e^{-1-\alpha/2}.
\end{split}
\end{equation}
Combining this with the estimates above for $I_1(x)$ and $I_2(x)$  yields  \eqref{l2-5-1}.

Note that, for any $x\in D_{4\e}$ and $s,\e\in (0,1)$, if
 $\delta_D(x+z)\le 3\e$, then it must hold that $|z|>{\delta_D(x)}/{4}$. By
 this observation, we obtain
\begin{equation}\label{l2-2-7}
\begin{split}
 \int_{\{|z|>s \e, \delta_D(x+z)\le 3\e\}}\frac{1}{|z|^{d+\alpha-1}}\,dz
&\le \int_{\{|z|>{\delta_D(x)}/{4}, \delta_D(x+z)\le 3\e\}}\frac{1}{|z|^{d+\alpha-1}}\,dz\\
&\le \sum_{k=0}^\infty \int_{\{{2^k\delta_D(x)}/{4}<|z|\le {2^{k+1}\delta_D(x)}/{4}, \delta_D(x+z)\le 3\e\}}
\frac{1}{|z|^{d+\alpha-1}}\,dz\\
&\le c_{18}\sum_{k=0}^\infty (2^k\delta_D(x))^{-d-\alpha+1}\int_{\{{2^k\delta_D(x)}/{4}<|z|\le {2^{k+1}\delta_D(x)}/{4}, \delta_D(x+z)\le 3\e\}}dz\\
&\le c_{19}\sum_{k=0}^\infty (2^k\delta_D(x))^{-d-\alpha+1}\left(2^{k+1}\delta_D(x)\right)^{d-1}\e\le c_{20}\e\delta_D(x)^{-\alpha}.
\end{split}
\end{equation}
Here in the fourth inequality above we have used the fact
\begin{align*}
\int_{\{{2^k\delta_D(x)}/{4}<|z|\le {2^{k+1}\delta_D(x)}/{4}, \delta_D(x+z)\le 3\e\}}dz\le c_{21}\e\left(2^{k+1}\delta_D(x)\right)^{d-1},
\end{align*}
which can be verified by the same argument for the estimate of $I_2(x)$ above.
This, along with \eqref{l2-5-6}, in turn implies \eqref{l2-5-2}.

(2) For simplicity we only prove \eqref{l2-5-3} and \eqref{l2-5-4} for the case that $\alpha=1$, since the case
$\alpha\in (0,1)$ can be proved almost in the same way.
Below in Step (2), we let $\alpha=1$ unless particularly stressed.
For every $x\in D_{4\e}$, it holds for every $s,\e \in (0,1)$ that
\begin{align*}
&  \int_{\{s\e<|z|\le s\}}\delta_D(x+z)^{-2+\alpha/2}\I_{\{x+z\in D_{3\e}\}}\frac{1}{|z|^{d+\alpha-1}}\,dz\\
&\le c_{22}\Bigg(\delta_D(x)^{-2+\alpha/2}\int_{\left\{s\e<|z|\le
\min\left({\delta_D(x)}/{4},s\right),
\delta_D(x+z)>3\e\right\}}\frac{1}{|z|^{d+\alpha-1}}\,dz\\
&\qquad\quad +\int_{\left\{
\min\left({\delta_D(x)}/{4},s\right)
<|z|\le s,\delta_D(x+z)>3\e\right\}}
\frac{\delta_D(x+z)^{-2+\alpha/2}}{|z|^{d+\alpha-1}}\,dz\Bigg)\\
&=:J_{1}(x)+J_{2}(x),
\end{align*}
where we have used \eqref{l2-5-7} in the inequality above for the term $J_1(x)$.
It is obvious that
\begin{align*}
J_{1}(x)\le c_{23}\delta_D(x)^{-2+\alpha/2}
(1+|\log \e|).
\end{align*}
Furthermore, following the argument for
the estimate
of $I_2(x)$ above,
we can obtain that
\begin{align*}
J_{2}(x)&
\le c_{24}\I_{\{s> \delta_D(x)/4\}}\int_{\left\{|z|>\delta_D(x)/{4},\delta_D(x+z)>3\e\right\}}\frac{\delta_D(x+z)^{-2+\alpha/2}}{|z|^{d+\alpha-1}}\,dz
\le c_{25}\e^{-1+\alpha/2}\delta_D(x)^{-\alpha}\I_{\{s> \delta_D(x)/4\}}.
\end{align*}

When $x\in D\backslash D_{4\e}$, we have, for any $s,\e\in (0,1)$,
\begin{equation}\label{l2-5-8}
\begin{split}
&\int_{\{s\e<|z|\le s \}}\left(\delta_D(x+z)^{-2+\alpha/2}\I_{\{x+z\in D_{3\e}\}}+\e^{-2+\alpha/2}\I_{\{x+z\in D \backslash D_{3\e}\}}\right)\frac{1}{|z|^{d+\alpha-1}}\,dz\\
&\le c_{26}\e^{-2+\alpha/2}\int_{\{s\e<|z|\le s\}}\frac{1}{|z|^{d+\alpha-1}}\,dz\le c_{27}
\e^{-2+\alpha/2}(1+|\log \e|).
\end{split}
\end{equation}
Combining this with the estimates above for $J_1(x)$ and  $J_2(x)$ yields  \eqref{l2-5-3}.

According to the argument of \eqref{l2-2-7}, we deduce that for every $x\in D_{4\e}$ and $s,\e \in (0,1)$,
\begin{align*}
&\int_{\{s\e<|z|\le s, \delta_D(x+z)\le 3\e\}}\frac{1}{|z|^{d+\alpha-1}}\,dz
\le \int_{\{|z|> \delta_D(x)/4, \delta_D(x+z)\le 3\e\}}\frac{1}{|z|^{d+\alpha-1}}\,dz
\le c_{28}\e\delta_D(x)^{-\alpha}.
\end{align*}
By this and \eqref{l2-5-8}, we prove \eqref{l2-5-4}.

For every $\alpha\in (0,1]$, $x\in D$ and
$s,\e\in (0,1)$, we have
\begin{align*}
&\int_{\{|z|> s \}}\left(\delta_D(x+z)^{-1+\alpha/2}\I_{\{x+z\in D_{3\e}\}}+\e^{-1+\alpha/2}\I_{\{x+z\in D \backslash D_{3\e}\}}\right)\frac{1}{|z|^{d+\alpha}}\,dz\\
&\le c_{29}\e^{-1+\alpha/2}\int_{\{|z|> s\}}\frac{1}{|z|^{d+\alpha}}\,dz\le c_{30}\e^{-1+\alpha/2}s^{-\alpha}.
\end{align*}
So \eqref{l2-5-5} is proved.

(3) Suppose that $\alpha\in (0,1)$. For every $\beta\in (\alpha/2,\alpha)$, $\e\in (0,1)$ and $x\in D_{4\e}$,  following the same argument as that for \eqref{l2-5-1}, we have
\begin{align*}
&\int_{\{|z|>\e\}}\delta_D(x+z)^{-1-\beta+\alpha/2}\I_{\{x+z\in D_{3\e}\}}\frac{1}{|z|^{d+\alpha-\beta}}\,dz\\
&\le c_{31}\Bigg(\delta_D(x)^{-1-\beta+\alpha/2}\int_{\{\e<|z|\le{\delta_D(x)}/{4}\}}\frac{1}{|z|^{d+\alpha-\beta}}\,dz+\int_{\{|z|>{\delta_D(x)}/{4},\delta_D(x+z)>3\e\}}
\frac{\delta_D(x+z)^{-1-\beta+\alpha/2}}{|z|^{d+\alpha-\beta}}\,dz\Bigg)\\
&=:L_{1}(x)+L_{2}(x).
\end{align*}
  Furthermore, one can see that
 $$
L_1(x)\le c_{32}\delta_D(x)^{-1-\beta+\alpha/2}\e^{-(\alpha-\beta)}.
$$
Similar to the
argument for the estimates of $I_2(x)$ above
we have
\begin{align*}
L_{2}(x)&\le \sum_{k=0}^{\infty}
\int_{\{{2^k\delta_D(x)}/{4}<|z|\le {2^{k+1}\delta_D(x)}/{4},\delta_D(x+z)>3\e\}}
\frac{\delta_D(x+z)^{-1-\beta+\alpha/2}}{|z|^{d+\alpha-\beta}}\,dz\\
&\le c_{33}\sum_{k=0}^\infty(2^k\delta_D(x))^{-d-\alpha+\beta}\int_{\{{2^k\delta_D(x)}/{4}<|z|\le {2^{k+1}\delta_D(x)}/{4},\delta_D(x+z)>3\e\}}
\delta_D(x+z)^{-1-\beta+\alpha/2}\,dz\\
&\le c_{34}\sum_{k=0}^\infty(2^k\delta_D(x))^{-d-\alpha+\beta}\left(2^{k+1}\delta_D(x)\right)^{d-1}
\int_{3\e}^{\infty}s^{-1-\beta+\alpha/2}\,ds\\
&\le c_{35}\e^{-(\beta-\alpha/2)}\delta_D(x)^{-1-(\alpha-\beta)}\sum_{k=0}^\infty 2^{-k(1+\alpha-\beta)}\le c_{36}\e^{-(\beta-\alpha/2)}
\delta_D(x)^{-1-(\alpha-\beta)}.
\end{align*}

When $x\in D\backslash D_{4\e}$ and $s,\e\in [0,1]$, it holds that
\begin{equation}\label{l2-5-9}
\begin{split}
&\int_{\{|z|>\e \}}\left(\delta_D(x+z)^{-1-\beta+\alpha/2}\I_{\{x+z\in D_{3\e}\}}+\e^{-1-\beta+\alpha/2}\I_{\{x+z\in D \backslash D_{3\e}\}}\right)\frac{1}{|z|^{d+\alpha-\beta}}\,dz\\
&\le c_{37}\e^{-1-\beta+\alpha/2}\int_{\{|z|>\e\}}\frac{1}{|z|^{d+\alpha-\beta}}\,dz\le c_{38}
\e^{-1-\alpha/2}.
\end{split}
\end{equation}
Combining this with the estimates above for $L_1(x)$ and $L_2(x)$ yields \eqref{l2-5-3a}.

Following the same argument as that for \eqref{l2-2-7}, we can show that
for every $x\in D_{4\e}$ and $s,\e\in [0,1]$,
\begin{align*}
&\int_{\{|z|>\e, \delta_D(x+z)\le 3\e\}}\frac{1}{|z|^{d+\alpha-\beta}}\,dz
\le c_{39}\e\delta_D(x)^{-1-(\alpha-\beta)}.\\
\end{align*}
By this and \eqref{l2-5-9},
we obtain \eqref{l2-5-4a}.
\end{proof}
\end{lemma}
\begin{remark}

In fact, when $\alpha\in (1,2)$, a direct calculation yields that
\begin{align*}
I_2(x)&=\int_{\{
x+z\in D:
|z|>{\delta_D(x)}/{4},\delta_D(x+z)>3\e\}}
\frac{\delta_D(x+z)^{-2+\alpha/2}}{|z|^{d+\alpha-1}}\,dz\\
&\le c_1\e^{-2+\alpha/2}\int_{\{z\in D:|z|>\delta_D(x)/4\}}
\frac{1}{|z|^{d+\alpha-1}}\,dz\le c_2\e^{-2+\alpha/2}\delta_D(x)^{-1+\alpha/2}.
\end{align*}
However, this is not enough for our application later. So here we need some more effort to obtain a
stronger estimate \eqref{l2-5-6a}.
\end{remark}

Let
$\eta_\e \in C_c^\infty(D)$
be a cut-off function near
the boundary $\partial D$ (which can be constructed
 by  using the $C^\infty$ regularizing
distance function $\tilde \delta_D$)
so that $0\leq \eta_\e \leq 1$ on $\R^d$,
\begin{equation}\label{e2-2}
\begin{split}
\eta_\e(x)=
\begin{cases}
1,\ \ & x\in D_{2\e},\\
 0,\ \ & x\in
\R^d \setminus  D_{\e},
\end{cases}
\end{split}
\end{equation} and
\begin{equation}\label{e2-1}
|\nabla^k\eta_\e(x)|\le c_0\e^{-k}\I_{D_{2\e}\backslash D_{\e}}(x),\quad k=1,2,3,4,\ x\in D,
\end{equation}
where the constant $c_0$ is independent of $\e$ and $x$.
We
have the following estimate for $\bar \LL \eta_\e$, where $\bar \LL$ is
the fractional Laplacian given by  \eqref{e:1.7}.

\begin{lemma}\label{l2-6}  Suppose that $D$ is a bounded Lipschitz domain  in $\R^d$.
Then there exists a constant $C_1>0$ such that
\begin{equation}\label{l2-6-1}
\left|\bar \LL \eta_\e(x)\right|\le C_1\left(\e\delta_D(x)^{-1-\alpha}\I_{D_{3\e}}(x)+\e^{-\alpha}\I_{D \backslash D_{3\e}}(x)\right),\quad x\in D.
\end{equation}
\end{lemma}
\begin{proof}
When $x\in D_{3\e}$, by \eqref{e2-2}, we know that $\nabla \eta_\e(x)=0$ and
$$\eta_\e(x+z)-\eta_\e(x)\neq 0 \ {\rm only\ if}
\ x+z\in D\backslash D_{2\e}.$$
Because $x\in D_{3\e}$ and $x+z\in D\backslash D_{2\e}$ imply that $|z|>\delta_D(x)/3$,
\begin{align*}
|\bar \LL \eta_\e(x)|&=\left|\int_{\R^d}\left(\eta_\e(x+z)-\eta_\e(x)-\langle \nabla \eta_\e(x), z\rangle \right)\frac{\bar K}{|z|^{d+\alpha}}\,dz\right|\\
&\le 2\int_{\{|z|>\delta_D(x)/3,\delta_D(x+z)\le 2\e\}}\frac{1}{|z|^{d+\alpha}}\,dz\le c_1\e\delta_D(x)^{-1-\alpha},
\end{align*}
where in the inequality we have used
\begin{align}\label{l2-6-2}
\int_{\{|z|>\delta_D(x)/3,\delta_D(x+z)\le 2\e\}}\frac{1}{|z|^{d+\alpha}}\,dz\le c_2\e\delta_D(x)^{-1-\alpha}
\end{align}
that can be proved by following the argument for \eqref{l2-2-7}.

When $x\in D\backslash D_{3\e}$, it follows from \eqref{e2-1}  that
\begin{align*}
|\bar \LL \eta_\e(x)|&\le \|\nabla^2 \eta_\e\|_\infty\int_{\{|z|\le \e\}}\frac{|z|^2}{|z|^{d+\alpha}}\,dz
+2\int_{\{|z|>\e\}}\frac{1}{|z|^{d+\alpha}}\,dz\le c_3\e^{-\alpha}.
\end{align*}
This establishes \eqref{l2-6-1}.
\end{proof}

\begin{lemma}\label{l2-7}
 Suppose that $D$ is a bounded Lipschitz domain  in $\R^d$.
Then, for every $f\in C^1_{(-\alpha/2)}(D)$, there exists a constant $C_1>0$ such that for all $\varepsilon\in (0,1)$ and $x\in D$,
\begin{equation}\label{l2-7-1}
\begin{split}
&\left|\int_{\R^d}\left(f(x+z)-f(x)\right)
\left(\eta_\e(x+z)-\eta_\e(x)\right)\frac{1}{|z|^{d+\alpha}}\,dz\right|\\
&\le
C_1\left(\e\delta_D(x)^{-1-\alpha/2}\I_{D_{4\e}}(x)+\e^{-\alpha/2}\I_{D\backslash D_{4\e}}(x)\right).
\end{split}
\end{equation}
\end{lemma}

\begin{proof}
Since $f\in C^1_{(-\alpha/2)}(D)\subset C^{\alpha/2}_{(-\alpha/2)}(D)$
(see e.g. \cite[Lemma 4.1]{ZZ} for
an
equivalent characterization of the space
$C^\beta_{(-\alpha/2)}(D)$)
and
$f(x)=0$ for every $x\in D^c$, it holds that
\begin{equation}\label{l2-7-2}
|f(x)-f(x+z)|\le c_1|z|^{\alpha/2},\quad x,z\in \R^d.
\end{equation}
and
\begin{equation}\label{l2-7-3}
|f(x)|\le c_1\delta_D(x)^{\alpha/2},\quad x\in D.
\end{equation}

When $x\in D_{4\e}$, due to \eqref{l2-7-3}
and the fact that
$$\eta_\e(x+z)-\eta_\e(x)
\  \hbox{ is possibly non-zero only if when }
\ x+z\in D\backslash D_{2\e},$$
we obtain
\begin{align*}
&\left|\int_{\R^d}\left(f(x+z)- f(x)\right)
\left(\eta_\e(x+z)-\eta_\e(x)\right)\frac{1}{|z|^{d+\alpha}}\,dz\right|\\
&\le c_2\int_{\{|z|>\delta_D(x)/3\}}
\left(\delta_D(x+z)^{\alpha/2}+\delta_D(x)^{\alpha/2}\right)\I_{\{x+z\in D\backslash D_{2\e}\}}\frac{1}{|z|^{d+\alpha}}\,dz\\
&\le c_3\left(\e^{\alpha/2}+\delta_D(x)^{\alpha/2}\right)\int_{\{|z|>\delta_D(x)/3,\delta_D(x+z)\le 2\e\}}\frac{1}{|z|^{d+\alpha}}\,dz
\le c_4\e\delta_D(x)^{-1-\alpha/2},
\end{align*}
where in the first inequality we used  the fact that $x\in D_{4\e}$ and $x+z\in D\backslash D_{2\e}$ imply $|z|>\delta_D(x)/3$, and
the last inequality follows from \eqref{l2-6-2}.

When $x\in D\backslash D_{4\e}$,
\begin{align*}
&\left|\int_{\R^d}\left(f(x+z)- f(x)\right)
\left(\eta_\e(x+z)-\eta_\e(x)\right)\frac{1}{|z|^{d+\alpha}}\,dz\right|\\
&\le \left[ \int_{\{|z|\le 5\e\}}+\int_{\{|z|>5\e\}}\right]\left|f(x+z)- f(x)\right|
\left|\eta_\e(x+z)-\eta_\e(x)\right|\frac{1}{|z|^{d+\alpha}}\,dz\\
&=:I_1^\e(x)+I_2^\e(x).
\end{align*}
Using \eqref{l2-7-2} and \eqref{e2-1}, we obtain
\begin{align*}
I_1^\e(x)&\le c_5 \|\nabla \eta_\e\|_\infty\int_{\{|z|\le 5\e\}}\frac{1}{|z|^{d+\alpha/2-1}}\,dz\le c_6\e^{-\alpha/2}
\end{align*} and
\begin{align*}
I_2^\e(x)&\le c_7\|\eta_\e\|_\infty\int_{\{|z|>5\e\}} \frac{1}{|z|^{d+\alpha/2}}\,dz\le c_8\e^{-\alpha/2}.
\end{align*}
Therefore, putting all the estimates above together yields \eqref{l2-7-1}.
\end{proof}

\subsection{Averaging estimates  related
to $\LL_\e f$}
Throughout this subsection,
without any mention we always assume that
$D$ is a bounded Lipschitz domain  in $\R^d$.
For any  $f\in C^{3}_{(-\alpha/2)}(D)$, set $f_\e(x):=f(x)\eta_\e(x)$, where
$\eta_\e\in C_c^\infty(D)$
is a cut-off function satisfying \eqref{e2-2} and \eqref{e2-1}.  By the definition, it holds for $f\in C^{3}_{(-\alpha/2)}(D)$ that
\begin{equation}\label{l2-2-3a}
\begin{split}
&|\nabla^k f(x)|\le  \|f\|_{3;D}^{(-\alpha/2)}\delta_D(x)^{-k+\alpha/2},\quad x\in D,\ k=0,1,2,3.
\end{split}
\end{equation}
Combining this with \eqref{e2-1} yields that
\begin{equation}\label{l2-2-3}
\begin{split}
|\nabla^k f_\e(x)|&\le c_0\delta_D(x)^{-k+\alpha/2}\I_{D_{\e}}(x)\\
&\le c_1\left(\delta_D(x)^{-k+\alpha/2}\I_{D_{3\e}}(x)+\e^{-k+\alpha/2}\I_{D\backslash D_{3\e}}(x)\right)
,\quad x\in D,\ k=0,1,2,3.
\end{split}
\end{equation}

For every $g\in
 C^2_b (\R^d)$,
define
\begin{equation}\label{l2-2-2}
\begin{split}
\hat \LL_{1,\e} g(x) :=&\int_{\R^d}\left(g(x+z)-g(x)-\left\langle \nabla g(x),z\right\rangle\right)\frac{K\left({x}/{\e},({x+z})/{\e}\right)}{|z|^{d+\alpha}}\,dz,\quad \alpha\in (1,2),\\
\hat \LL_{2,\e} g(x) :=&\int_{\R^d}\left(g(x+z)-g(x)-\left\langle \nabla g(x),z\right\rangle\I_{\{|z|\le 1\}}\right)\frac{K\left({x}/{\e},({x+z})/{\e}\right)}{|z|^{d+\alpha}}\,dz,\quad \alpha\in (0,1];\\
\end{split}
\end{equation} and
\begin{equation}\label{l2-2-2a}
\begin{split}
\bar \LL_{\e} g(x) :=&\begin{cases}\displaystyle\int_{\R^d}\left(g(x+z)-g(x)-\left\langle \nabla g(x),z\right\rangle \right)\frac{\bar K\left({x}/{\e}\right)}{|z|^{d+\alpha}}\,dz,\quad \alpha\in (1,2),\\
\displaystyle\int_{\R^d}\left(g(x+z)-g(x)-\left\langle \nabla g(x),z\right\rangle\I_{\{|z|\le 1\}} \right)\frac{\bar K\left({x}/{\e}\right)}{|z|^{d+\alpha}}\,dz,\quad \alpha\in (0,1],\end{cases}
\end{split}
\end{equation}
where $\bar K(x):=\displaystyle\int_{\T^d}K(x,y)\,dy$.
(Note that the constant $\bar K$ defined in \eqref{e:1.7a} is the value of $\int_{\T^d} \bar K(x) \,dx$.)

The purpose of this part is to present several
estimates concerning the averaging properties for $f\in C^{3}_{(-\alpha/2)}(D)$. For this, we need to take account
into the effect of blow up behaviours for $\nabla f$ and $\nabla^2 f$ near the boundary $\partial D$.

\begin{lemma}\label{l2-2}
Suppose that $\alpha\in (1,2)$. Then, for every $f\in C^{3}_{(-\alpha/2)}(D)$, there are constants $C_1,C_2>0$ such that for all $\e\in (0,1)$ and $x\in D$,
\begin{equation}\label{l2-2-1}
\begin{split}
 |\hat \LL_{1,\e} f_\e(x)-\bar \LL_{\e} f_\e(x)|
&\le C_1\left(\e^{2-\alpha}\delta_D(x)^{-2+\alpha/2}+\e^{\alpha/2}\delta_D(x)^{-\alpha}\right)\I_{D_{4\e}}(x)
+C_1\e^{-\alpha/2}\I_{D \backslash D_{4\e}}(x)
\end{split}
\end{equation} and
\begin{equation}\label{l2-2-1a}
\begin{split}
 |\hat \LL_{1,\e} \nabla f_\e(x)|
&\le C_2\left(\e^{1-\alpha}\delta_D(x)^{-2+\alpha/2}+\e^{-1+\alpha/2}\delta_D(x)^{-\alpha}\right)\I_{D_{4\e}}(x)
+C_2\e^{-1-\alpha/2}\I_{D \backslash  D_{4\e}}(x).
\end{split}
\end{equation}
Here $f_\e := f\eta_\e$.
 \end{lemma}

\begin{proof} (1) Set
\begin{align*}
H_\e(x,z):=\left(f_\e(x+z)-f_\e(x)-\left\langle \nabla f_\e(x), z\right\rangle\right) |z|^{-(d+\alpha)}.
\end{align*}
For $z=(z_1,\cdots,z_d)\in \R^d$, let $|z|_1:=\max_{1\le i \le d}|z_i|$.
Write
\begin{align*}
\quad  |\hat \LL_{1,\e} f_\e(x)-\bar \LL_{\e} f_\e(x) |\le& \int_{\{|z|_1\le 2\e\}}|H_\e(x,z)|K\left(x/\e,(x+z)/{\e}\right)\,dz+
\int_{\{|z|_1\le 2\e\}}|H_\e(x,z)|\bar K\left(x/\e\right)\,dz\\
&+
\left|\int_{\{|z|_1>2\e\}}H_\e(x,z)\left(K\left(x/\e,(x+z)/{\e}\right)-\bar K\left(x/\e\right)\right)\,dz\right|\\
=&:I_1^\e(x)+I_2^\e(x)+I_3^\e(x).
\end{align*}

By the mean-value theorem and \eqref{l2-2-3}, we obtain
\begin{align*}
I_1^\e(x)+I_2^\e(x)&\le c_1\sup_{y:|y-x|_1\le 2\e}|\nabla^2 f_\e(y)|\cdot\left(\int_{\{|z|_1\le 2\e\}}\frac{|z|^2}{|z|^{d+\alpha}}\,dz\right)\\
&\le c_2\e^{2-\alpha}\left(\delta_D(x)^{-2+\alpha/2}\I_{D_{4\e}}(x)+\e^{-2+\alpha/2}
\I_{D\backslash D_{4\e}}(x)\right).
\end{align*}

On the other hand, set
\begin{align*}
&\int_{\{|z|_1>2\e\}}H_\e(x,z)\left(K\left(x/\e,(x+z)/{\e}\right)-\bar K\left(x/\e\right)\right)dz\\
&=\sum_{y\in \e\Z^d:|y|_1>2\e}\int_{(y,y+\e e]}\left(H_\e(x,z)-H_\e(x,y)\right)K\left(x/\e,(x+z)/{\e}\right)dz\\
&\quad +\sum_{y\in \e\Z^d:|y|_1>2\e}\int_{(y,y+\e e]}\left(H_\e(x,y)-H_\e(x,z)\right)\bar K\left(x/\e\right)dz,
\end{align*}
where $e=(1,\cdots,1)$, $(y,y+\e e]:=\prod_{i=1}^d(y_i,y_i+\e]$ and we
have used the following property
\begin{align*}
\int_{(y,y+\e e]}H_\e(x,y)K\left(x/\e,(x+z)/ \e\right)dz=
H_\e(x,y)\bar K\left(x/\e\right)\e^d
\quad y\in \e\Z^d\ {\rm with}\ |y|_1>2\e,
\end{align*}
thanks to the fact that $z\mapsto K(x,z)$ is $1$-periodic. Hence,
\begin{equation}\label{l2-2-4}
\begin{split}
I_3^\e(x)&\le \sum_{y\in \e\Z^d:|y|_1>2\e}\left|\int_{(y,y+\e e]}\left(H_\e(x,z)-H_\e(x,y)\right)K\left(x/\e,(x+z)/{\e}\right)\,dz\right|\\
&\quad +\sum_{y\in \e\Z^d:|y|_1>2\e}\left|\int_{(y,y+\e e]}\left(H_\e(x,y)-H_\e(x,z)\right)\bar K\left(x/\e\right)\,dz\right|\\
&\le c_3\sum_{y\in \e\Z^d:|y|_1>2\e}\int_{(y,y+\e e]}\left|H_\e(x,z)-H_\e(x,y)\right|dz.
\end{split}
\end{equation}

According to the mean-value theorem, for every $y\in \e\Z^d$ with $|y|_1>2\e$ and $z\in (y,y+\e e]$,
\begin{align*}
&|H_\e(x,z)-H_\e(x,y)|\\
&\le \left|f_\e(x+z)-f_\e(x)-\left\langle \nabla f_\e(x), z\right\rangle\right|\cdot \left|\frac{1}{|y|^{d+\alpha}}-
\frac{1}{|z|^{d+\alpha}}\right|\\
&\quad +\left|\left(f_\e(x+z)-f_\e(x)-\left\langle \nabla f_\e(x), z\right\rangle\right)-
\left(f_\e(x+y)-f_\e(x)-\left\langle \nabla f_\e(x), y\right\rangle\right)\right|\cdot \frac{1}{|y|^{d+\alpha}}\\
&\le c_4\e \left[ \left(\int_0^1\int_0^1 |\nabla^2 f_\e(x+stz)|\,ds\,dt\right)\cdot \frac{|z|^2}{|z|^{d+\alpha+1}}+
\left(\int_0^1\left|\nabla f_\e(x+sy+(1-s)z)-\nabla f_\e(x)\right|\,ds\right)\cdot \frac{1}{|z|^{d+\alpha}}\right]\\
&\le c_5\e\left[\left(\int_0^1\int_0^1 |\nabla^2 f_\e(x+stz)|\,ds\,dt+
\int_0^1\int_0^1|\nabla^2 f_\e\left(x+t(sy+(1-s)z)\right)|\,ds\,dt\right)\cdot \frac{1}{|z|^{d+\alpha-1}}\right]\\
&\le c_6\e\Bigg[\int_0^1\int_0^1\Big(\delta_D(x+stz)^{-2+\alpha/2}\I_{\{x+stz\in D_{3\e}\}}+
\delta_D(x+tz)^{-2+\alpha/2}\I_{\{x+tz\in D_{3\e}\}}\\
&\qquad\qquad \qquad\quad +\e^{-2+\alpha/2}\I_{\{x+stz\in D\backslash D_{3\e}\}}+\e^{-2+\alpha/2}\I_{\{x+tz\in D\backslash D_{3\e}\}}\Big)\,ds\,dt\Bigg]\cdot \frac{1}{|z|^{d+\alpha-1}},
\end{align*}
where in the last inequality we have used \eqref{l2-2-3} and the fact that $y\in \e\Z^d$ with $|y|_1>2\e$ and $z\in (y,y+\e e]$.
Thus,
\begin{align*}
& I_{3}^\e(x)\\
&\le c_7\e\sum_{y\in \e\Z^d:|y|_1>2\e}\int_0^1\int_0^1
\int_{(y,y+\e e]}\Big(\delta_D(x+stz)^{-2+\alpha/2}\I_{\{x+stz\in D_{3\e}\}}+
\delta_D(x+tz)^{-2+\alpha/2}\I_{\{x+tz\in D_{3\e}\}}\\
&\qquad\qquad \qquad\qquad\qquad \qquad\qquad \quad\,\,\,+\e^{-2+\alpha/2}\I_{\{x+stz\in D\backslash D_{3\e}\}}
+\e^{-2+\alpha/2}\I_{\{x+tz\in D\backslash D_{3\e}\}}\Big)\frac{1}{|z|^{d+\alpha-1}}\,dz\,ds\,dt\\
&\le c_8\e\int_0^1\int_0^1\int_{\{|z|_1>\e\}}\Big(\delta_D(x+stz)^{-2+\alpha/2}\I_{\{x+stz\in D_{3\e}\}}+\delta_D(x+tz)^{-2+\alpha/2}\I_{\{x+tz\in D_{3\e}\}}\\
&\qquad\qquad \qquad\qquad \qquad+\e^{-2+\alpha/2}\I_{\{x+stz\in D \backslash D_{3\e}\}}+\e^{-2+\alpha/2}\I_{\{x+tz\in D\backslash D_{3\e}\}}\Big)\frac{1}{|z|^{d+\alpha-1}}\,dz\,ds\,dt\\
&\le c_{9}\e\Big(\int_0^1\int_0^1 (st)^{\alpha-1}\int_{\{|z|_1>st \e\}}\left(\delta_D(x+z)^{-2+\alpha/2}\I_{\{x+z\in D_{3\e}\}}+\e^{-2+\alpha/2}\I_{\{x+z\in D\backslash D_{3\e}\}}\right)\frac{1}{|z|^{d+\alpha-1}}\,dz\,ds\,dt\\
&\qquad\quad\quad   + \int_0^1\int_0^1 t^{\alpha-1}\int_{\{|z|_1>t \e\}}\left(\delta_D(x+z)^{-2+\alpha/2}\I_{\{x+z\in D_{3\e}\}}+\e^{-2+\alpha/2}\I_{\{x+z\in D\backslash D_{3\e}\}}\right)\frac{1}{|z|^{d+\alpha-1}}\,dz\,ds\,dt\Big),
\end{align*}
where in the last inequality we have used
a
change of variables. Therefore, according to \eqref{l2-5-1} and \eqref{l2-5-2}, (noting that although here we use the norm $|z|_1$ in place of $|z|$, it is easy to verify that both \eqref{l2-5-1} and \eqref{l2-5-2} still hold since $|z|_1$ is equivalent with $|z|$), we obtain
\begin{align*}
I_{3}^\e(x)&\le c_{10}\left(\e^{2-\alpha}\delta_D(x)^{-2+\alpha/2}+\e^{\alpha/2}\delta_D(x)^{-\alpha}\right)\I_{D_{4\e}}(x)
+c_{10}\e^{-\alpha/2}\I_{D \backslash D_{4\e}}(x).
\end{align*}
So, putting all the estimates for $I_1^\e(x)$, $I_2^\e(x)$ and $I_3^\e(x)$ together,
we obtain
\eqref{l2-2-1}.

(2) Define
\begin{align*}
G_\e(x,z):=\left(\nabla f_\e(x+z)-\nabla f_\e(x)-\left\langle \nabla^2 f_\e(x), z\right\rangle\right)\frac{1}{|z|^{d+\alpha}}.
\end{align*}
By the mean-value theorem, we have
\begin{align*}
|\hat \LL_{1,\e} \nabla f_\e(x)|&\le \int_{\{|z|\le 2\e\}}|G_\e(x,z)|K\left(x/\e,(x+z)/{\e}\right)\,dz+
 \int_{\{|z|>2\e\}}|G_\e(x,z)|K\left(x/\e,(x+z)/{\e}\right)\,dz.\\
&=:J_1^\e(x)+J_2^\e(x).
\end{align*}
According to \eqref{l2-2-3},
\begin{align*}
J_1^\e(x)&\le c_{11}\left(\sup_{y:|y-x|\le 2\e}|\nabla^3 f_\e(y)|\right)\cdot\int_{\{|z|\le 2\e\}}\frac{|z|^2}{|z|^{d+\alpha}}\,dz\\
&\le c_{12}\e^{2-\alpha}\left(\delta_D(x)^{-3+\alpha/2}\I_{D_{4\e}}(x)+\e^{-3+\alpha/2}\I_{D\backslash D_{4\e}}(x)\right).
\end{align*}
On the other hand, by the mean-value theorem again and \eqref{l2-2-3}, we derive
\begin{align*}
J_2^\e(x)&\le c_{13}\int_0^1\int_{\{|z|>2\e\}}\left|\nabla^2 f_\e(x+sz)\right|\frac{1}{|z|^{d+\alpha-1}}\,dz\,ds+
c_{13}|\nabla^2 f_\e(x)|\cdot \int_{\{|z|>2\e\}}\frac{1}{|z|^{d+\alpha-1}}\,dz\\
&\le c_{14}\int_0^1 \int_{\{|z|>2\e\}}\left(\delta_D(x+sz)^{-2+\alpha/2}\I_{\{x+sz\in D_{3\e}\}}+\e^{-2+\alpha/2}\I_{\{x+sz\in D\backslash D_{3\e}\}} \right)\frac{1}{|z|^{d+\alpha-1}}\,dz\,ds\\
&\quad +
c_{14}\e^{1-\alpha}\delta_{D}(x)^{-2+\alpha/2}\I_{D_{4\e}}(x)+c_{16}\e^{-1-\alpha/2}\I_{D \backslash D_{4\e}}(x)\\
&\le c_{15}\int_0^1\int_0^1 s^{\alpha-1}\int_{\{|z|>s \e\}}\left(\delta_D(x+z)^{-2+\alpha/2}\I_{\{x+z\in D_{3\e}\}}+\e^{-2+\alpha/2}\I_{\{x+z\in D\backslash D_{3\e}\}}\right)\frac{1}{|z|^{d+\alpha-1}}\,dz\,ds\,dt\\
&\quad +c_{15}\e^{1-\alpha}\delta_{D}(x)^{-2+\alpha/2}\I_{D_{4\e}}(x)+c_{17}\e^{-1-\alpha/2}\I_{D \backslash D_{4\e}}(x).
\end{align*}
Then by \eqref{l2-5-1} and \eqref{l2-5-2}
we get
\begin{align*}
J_2^\e(x)\le c_{18}\left(\e^{1-\alpha}\delta_D(x)^{-2+\alpha/2}+\e^{-1+\alpha/2}\delta_D(x)^{-\alpha}\right)\I_{D_{4\e}}(x)
+c_{18}\e^{-1-\alpha/2}\I_{D \backslash D_{4\e}}(x).
\end{align*}

Putting
all the estimates above together,
we obtain
the second assertion \eqref{l2-2-1a}.
\end{proof}

\begin{lemma}\label{l2-3}
Suppose that $\alpha\in (0,1]$. Then, for every $f\in C^{3}_{(-\alpha/2)}(D)$, there are positive constants $C_1$,
$C_2$ and $C_3$
such that for all $\e\in (0,1)$ and $x\in D$,
\begin{equation}\label{l2-3-1}
\begin{split}
 \left|\hat \LL_{2,\e} f_\e(x)-\bar \LL_{\e} f_\e(x)\right|
&\le C_1\left(\e^{\alpha/2}
\delta_D(x)^{-\alpha}
+\e\delta_D(x)^{-3/2}\left(1+|\log \e|\right)\I_{\{\alpha=1\}}\right)\I_{D_{4\e}}(x)\\
&\quad +C_1\e^{-1+\alpha/2}\left(1+|\log \e|\I_{\{\alpha=1\}}\right)\I_{D \backslash D_{4\e}}(x).
\end{split}
\end{equation}
Here $f_\e = f \eta_\e$. Moreover,
when $\alpha=1$,
\begin{equation}\label{l2-3-1a}
\begin{split}
 \left|\hat \LL_{2,\e} \nabla f_\e(x)\right|
&\le C_2\left(\delta_D(x)^{-3/2}\left(1+|\log \e|\right) +\e^{-1/2}
\delta_D(x)^{-1}\right)
\I_{D_{4\e}}(x)\\
&\quad+C_2\e^{-3/2}(1+|\log \e|)\I_{D \backslash  D_{4\e}}(x),
\end{split}
\end{equation}
and,  when $\alpha\in (0,1)$, for every $\beta\in (\alpha/2,\alpha)$, it holds that
\begin{equation}\label{l2-3-2}
\begin{split}
 \left|\LL_{\e} \nabla f_\e(x)\right|
&\le C_3\left(\e^{-(\beta-\alpha/2)}\delta_D(x)^{-1-(\alpha-\beta)}+\e^{-(\alpha-\beta)}\delta_D(x)^{-1-\beta+\alpha/2}\right)\I_{D_{4\e}}(x)+
C_3\e^{-1-\alpha/2}\I_{D\backslash D_{4\e}}(x).
\end{split}
\end{equation}
\end{lemma}
\begin{proof}
(i) Set \begin{align*}
H_\e(x,z):=\left(f_\e(x+z)-f_\e(x)-\left\langle \nabla f_\e(x), z\right\rangle\I_{\{|z|\le 1\}}\right)\frac{1}{|z|^{d+\alpha}}.
\end{align*}
Then
\begin{align*}
  |\hat \LL_{2,\e} f_\e(x)-\bar \LL_{\e} f_\e(x) |\le& \int_{\{|z|_1\le 2\e\}}|H_\e(x,z)|K\left(x/\e,(x+z)/{\e}\right)\,dz+
\int_{\{|z|_1\le 2\e\}}|H_\e(x,z)|\bar K\left(x/\e\right)\,dz\\
&+
\left|\int_{\{|z|_1>2\e\}}H_\e(x,z)\left(K\left(x/\e,(x+z)/{\e}\right)-\bar K\left(x/\e\right)\right)\,dz\right|\\
=&:I_1^\e(x)+I_2^\e(x)+I_3^\e(x).
\end{align*}
According to the mean-value theorem and \eqref{l2-2-3}, we have
\begin{align*}
I_1^\e(x)+I_2^\e(x)&\le c_1\sup_{y:|y-x|\le 2\e}|\nabla^2 f_\e(y)|\cdot\left(\int_{\{|z|\le 2\e\}}\frac{|z|^2}{|z|^{d+\alpha}}\,dz\right)\\
&\le c_2\e^{2-\alpha}\left(\delta_D(x)^{-2+\alpha/2}\I_{D_{4\e}}(x)+\e^{-2+\alpha/2}
\I_{D\backslash D_{4\e}}(x)\right).
\end{align*}
On the other hand, following the same argument as that for \eqref{l2-2-4}, we obtain \begin{align*}
I_3^\e(x)
\le c_3\sum_{y\in \e\Z^d:|y|_1>2\e}\int_{(y,y+\e e]}\left|H_\e(x,y)-H_\e(x,z)\right|\,dz.
\end{align*}
Applying the mean-value theorem and \eqref{l2-2-3} again
(by
following the same argument as that in the proof of Lemma \ref{l2-2}),
we obtain that for every $x\in D$, $y\in \e\Z^d$ with $|y|_1>2\e$ and $z\in (y,y+\e e]$,
\begin{align*}
&|H_\e(x,z)-H_\e(x,y)|\\ &\le
\left|f_\e(x+z)-f_\e(x)-\langle \nabla f_\e(x), z\rangle \I_{\{|z|\le 1- \e \}}\right|\left|\frac{1}{|y|^{d+\alpha}}-\frac{1}{|z|^{d+\alpha}}\right|\\
&\qquad  +\left|\left(f_\e(x+z)-f_\e(x)-\langle \nabla f_\e(x), z\rangle\I_{\{|z|\le 1- \e\}}\right)-\left(f_\e(x+y)-f_\e(x)-
\langle \nabla f_\e(x), y\rangle\I_{\{|z|\le 1- \e\}}\right)\right|
\frac{1}{|y|^{d+\alpha}}\\
&\qquad +\frac{\left|\langle \nabla f_\e(x), z\rangle\right|}{|z|^{d+\alpha}}\I_{\{1-\e\le |z|\le1\}}+ \frac{\left|\langle\nabla f_\e(x), y\rangle\right|}{|y|^{d+\alpha}}\I_{\{1-(1+\sqrt{d})\e\le |y|\le 1+(1+\sqrt{d})\e\}}
\\&\le c_4\e\left(\int_0^1\int_0^1 |\nabla^2 f_\e(x+stz)|\,ds\,dt+
\int_0^1\int_0^1|\nabla^2 f_\e\left(x+t(sy+(1-s)z)\right)|\,ds\,dt\right)\cdot \frac{1}{|z|^{d+\alpha-1}}\I_{\{|z|\le 1\}}\\
&\quad +c_4\e\left(\int_0^1 |\nabla f_\e(x+sz)|\,ds+\int_0^1\left|\nabla f_\e\left(x+sy+(1-s)z\right)\right|\,ds\right)\cdot \frac{1}{|z|^{d+\alpha}}
\I_{\{|z|> 1\}}\\
&\quad +c_4|\nabla f_\e(x)|\left(\I_{\{1-\e\le |z|\le1\}}+\I_{\{1-(1+\sqrt{d})\e\le |y|\le 1+(1+\sqrt{d})\e\}}\right)\\
&\le c_5\e\bigg[\int_0^1\int_0^1\Big(\delta_D(x+stz)^{-2+\alpha/2}\I_{\{x+stz\in D_{3\e}\}}+
\delta_D(x+tz)^{-2+\alpha/2}\I_{\{x+tz\in D_{3\e}\}}
+\e^{-2+\alpha/2}\I_{\{x+stz\in D\backslash D_{3\e}\}}\\
&\qquad \quad+\e^{-2+\alpha/2}\I_{\{x+tz\in D\backslash D_{3\e}\}}\Big)\,ds\,dt\bigg]\cdot \frac{1}{|z|^{d+\alpha-1}}\I_{\{|z|\le 1\}}
+c_5\e\bigg[\int_0^1 \Big(\delta_D(x+sz)^{-1+\alpha/2}\I_{\{x+sz\in D_{3\e}\}}\\
&\qquad\quad +\e^{-1+\alpha/2}\I_{\{x+sz\in D\backslash D_{3\e}\}}+\delta_D(x+z)^{-1+\alpha/2}\I_{\{x+z\in D_{3\e}\}}
+\e^{-1+\alpha/2}\I_{\{x+z\in D\backslash D_{3\e}\}}\Big)\,ds\bigg]\cdot\frac{1}{|z|^{d+\alpha}}\I_{\{|z|>1\}}\\
&\quad +c_5\left(\delta_D(x)^{-1+\alpha/2}
\I_{D_{3\e}}(x)
+\e^{-1+\alpha/2}
\I_{D\backslash D_{3\e}}(x)
\right)
\left(\I_{\{1-\e\le |z|\le1\}}+\I_{\{1-(1+\sqrt{d})\e\le |y|\le 1+(1+\sqrt{d})\e\}}\right).
\end{align*}
This, along with the proof
of the estimate for $I_3^\e(x)$ in Lemma \ref{l2-2}
via a change of variables,
yields that
\begin{align*}
& I_{3}^\e(x)\\
&\le c_6\e\bigg[\int_0^1\int_0^1 (st)^{\alpha-1}\int_{\{st\e<|z|_1\le st\}}\left(\delta_D(x+z)^{-2+\alpha/2}\I_{\{x+z\in D_{3\e}\}}+\e^{-2+\alpha/2}\I_{\{x+z\in D\backslash D_{3\e}\}}\right)\frac{1}{|z|^{d+\alpha-1}}\,dz\,ds\,dt\\
&\qquad\quad+ \int_0^1\int_0^1 s^{\alpha-1}\int_{\{s\e<|z|_1\le s\}}\left(\delta_D(x+z)^{-2+\alpha/2}\I_{\{x+z\in D_{3\e}\}}+\e^{-2+\alpha/2}\I_{\{x+z\in D\backslash D_{3\e}\}}\right)\frac{1}{|z|^{d+\alpha-1}}\,dz\,ds\,dt\bigg]\\
&\quad +c_6\e\bigg[\int_0^1 s^{\alpha}\int_{\{|z|> s\}}\left(\delta_D(x+z)^{-1+\alpha/2}\I_{\{x+z\in D_{3\e}\}}+\e^{-1+\alpha/2}\I_{\{x+z\in D\backslash D_{3\e}\}}\right)\frac{1}{|z|^{d+\alpha}}\,dz\,ds\\
&\qquad\qquad +\int_{\{|z|>1\}}\left(\delta_D(x+z)^{-1+\alpha/2}\I_{\{x+z\in D_{3\e}\}}+\e^{-1+\alpha/2}\I_{\{x+z\in D\backslash D_{3\e}\}}\right)\frac{1}{|z|^{d+\alpha}}\,dz\bigg]\\
&\quad +c_6\e\left(\delta_D(x)^{-1+\alpha/2}\I_{D_{3\e}}(x)+\e^{-1+\alpha/2}\I_{D\backslash D_{3\e}}(x)\right).
\end{align*}

Now, by \eqref{l2-5-3}--\eqref{l2-5-5}, we have
\begin{align*}
I_{3}^\e(x)&\le c_{7}\left(\e(\delta_D(x)^{-1-\alpha/2}+\delta_D(x)^{-2+\alpha/2}(1+|\log \e|)\I_{\{\alpha=1\}})+\e^{\alpha/2}(\delta_D(x)^{-\alpha}+1)\right)\I_{D_{4\e}}(x)\\
&\quad +c_{7}\left(\e^{-1+\alpha/2}(1+|\log \e|\I_{\{\alpha=1\}})+\e^{\alpha/2}\right)\I_{D \backslash D_{4\e}}(x)\\
&\le c_{8}\left(\e^{\alpha/2}
\delta_D(x)^{-\alpha}
+\e\delta_D(x)^{-3/2}\left(1+|\log \e|\right)\I_{\{\alpha=1\}}\right)\I_{D_{4\e}}(x)\\
&\quad +c_{8}\e^{-1+\alpha/2}\left(1+|\log \e|\I_{\{\alpha=1\}}\right)\I_{D \backslash D_{4\e}}(x),
\end{align*}
where in the last inequality we used the fact that $\sup_{x\in D}\delta_D(x)<\infty$ since $D$ is bounded.
Putting all the estimates above for $I_1^\e(x)$, $I_2^\e(x)$, $I_3^\e(x)$ together, we arrive at
\begin{align*}
&|\hat \LL_{2,\e} f_\e(x)-\bar \LL_{\e} f_\e(x) |\\
&\le c_{9}\left(\e^{\alpha/2}
\delta_D(x)^{-\alpha}
+\e^{2-\alpha}\delta_D(x)^{-2+\alpha/2}+
\e\delta_D(x)^{-3/2}\left(1+|\log \e|\right)\I_{\{\alpha=1\}}\right)\I_{D_{4\e}}(x)\\
&\quad +c_{9}\left(\e^{-1+\alpha/2}\left(1+|\log \e|\I_{\{\alpha=1\}}\right)+\e^{-\alpha/2}\right)\I_{D \backslash D_{4\e}}(x)\\
&\le c_{10}\left(\e^{\alpha/2}
\delta_D(x)^{-\alpha}
+\e\delta_D(x)^{-3/2}\left(1+|\log \e|\right)\I_{\{\alpha=1\}}\right)\I_{D_{4\e}}(x)\\
&\quad +c_{10}\e^{-1+\alpha/2}\left(1+|\log \e|\I_{\{\alpha=1\}}\right)\I_{D \backslash D_{4\e}}(x).
\end{align*}
So \eqref{l2-3-1} is proved.

(ii) When $\alpha=1$, define
\begin{align*}
G_\e(x,z):=\left(\nabla f_\e(x+z)-\nabla f_\e(x)-\left\langle \nabla^2 f_\e(x), z\right\rangle\I_{\{|z|\le 1\}}\right)\frac{1}{|z|^{d+1}}.
\end{align*}
Then,
\begin{align*}
|\hat \LL_{2,\e} \nabla f_\e(x)|&\le \int_{\{|z|\le 2\e\}}|G_\e(x,z)|K\left(x/\e,(x+z)/{\e}\right)\,dz+
 \int_{\{|z|>2\e\}}|G_\e(x,z)|K\left(x/\e,(x+z)/{\e}\right)\,dz.\\
&=:J_1^\e(x)+J_2^\e(x).
\end{align*}
Thanks to  \eqref{l2-2-3}, we obtain
\begin{align*}
J_1^\e(x)&\le c_{11}\left(\sup_{y:|y-x|\le 2\e}|\nabla^3 f_\e(y)|\right)\cdot\int_{\{|z|\le 2\e\}}\frac{|z|^2}{|z|^{d+1}}\,dz\\
&\le c_{12}\e\left(\delta_D(x)^{-5/2}\I_{D_{4\e}}(x)+\e^{-5/2}\I_{D\backslash D_{4\e}}(x)\right).
\end{align*}
On the other hand, by the mean-value theorem again and \eqref{l2-2-3}, we derive
\begin{align*}
J_2^\e(x)&\le c_{13}\int_0^1\int_{\{2\e<|z|\le 1\}}\left|\nabla^2 f_\e(x+sz)\right|\frac{1}{|z|^{d}}\,dz\,ds+
c_{13}|\nabla^2 f_\e(x)|\cdot \int_{\{2\e<|z|\le 1\}}\frac{1}{|z|^{d}}\,dz\\
&\quad +c_{13}\int_{\{|z|>1\}}\left|\nabla f_\e(x+z)\right|\frac{1}{|z|^{d+1}}\,dz+c_{13}|\nabla f_\e(x)|\cdot \int_{\{|z|>1\}}\frac{1}{|z|^{d+1}}\,dz\\
&\le c_{14}\int_0^1 \int_{\{2\e<|z|\le 1\}}\left(\delta_D(x+sz)^{-3/2}\I_{\{x+sz\in D_{3\e}\}}+\e^{-3/2}\I_{\{x+sz\in D\backslash D_{3\e}\}} \right)\frac{1}{|z|^{d}}\,dz\,ds\\
&\quad +c_{14}\int_{\{|z|>1\}}\left(\delta_D(x+z)^{-1/2}\I_{\{x+z\in D_{3\e}\}}+\e^{-1/2}\I_{\{x+z\in D\backslash D_{3\e}\}}\right)
\frac{1}{|z|^{d+1}}\,dz\\
&\quad +
c_{14}\left(1+|\log \e|\right)\left(\delta_{D}(x)^{-3/2}\I_{D_{4\e}}(x)+\e^{-3/2}\I_{D \backslash D_{4\e}}(x)\right)\\
&\quad +
c_{14}\left(\delta_{D}(x)^{-1/2}\I_{D_{4\e}}(x)+\e^{-1/2}\I_{D \backslash D_{4\e}}(x)\right)\\
&\le c_{15}\int_0^1  \int_{\{2\e s<|z|\le s\}}\left(\delta_D(x+z)^{-3/2}\I_{\{x+z\in D_{3\e}\}}+\e^{-3/2}\I_{\{x+z\in D\backslash D_{3\e}\}} \right)\frac{1}{|z|^{d}}\,dz\,ds\\
&\quad +c_{15}\int_{\{|z|>1\}}\left(\delta_D(x+z)^{-1/2}\I_{\{x+z\in D_{3\e}\}}+\e^{-1/2}\I_{\{x+z\in D\backslash D_{3\e}\}}\right)
\frac{1}{|z|^{d+1}}dz\\
&\quad +
c_{15}\left(1+|\log \e|\right)\left(\delta_{D}(x)^{-3/2}\I_{D_{4\e}}(x)+\e^{-3/2}\I_{D \backslash D_{4\e}}(x)\right)\\
&\quad +c_{15}\left(\delta_{D}(x)^{-1/2}\I_{D_{4\e}}(x)+\e^{-1/2}\I_{D \backslash D_{4\e}}(x)\right).
\end{align*}
Hence,  applying \eqref{l2-5-3} -- \eqref{l2-5-5}, we
get
\begin{align*}
J_2^\e(x)&\le c_{16}\left(\delta_D(x)^{-3/2}\left(1+|\log \e|\right)+\e^{-1/2}
\delta_D(x)^{-1}
\right)\I_{D_{4\e}}(x)
 +c_{16}\e^{-3/2}\left(1+|\log \e|\right)\I_{D \backslash D_{4\e}}(x).
\end{align*}
Putting both estimates for $J_1(x)$ and $J_2(x)$ together, we obtain the desired conclusion \eqref{l2-3-1a}.

(iii) Note that, when $\alpha\in (0,1)$, $f\in C^3_{(-\alpha/2)}(D)\subset C^{1+\beta}_{(-\alpha/2)}(D)$ for every $\beta\in (\alpha/2,\alpha)$. By  \eqref{e2-1} and the definition of $f_\e$, it holds that for all $x,z\in \R^d$ with $|z|\ge 2\e$
\begin{equation}\label{l2-3-3}
\begin{split}
\frac{\left|\nabla f_\e(x)-\nabla f_\e(x+z)\right|}{|z|^\beta} &\le c_{17}\Big[\delta_D(x+z)^{-1-\beta+\alpha/2}\I_{\{x+z\in D_{2\e}\}}+
\delta_D(x)^{-1-\beta+\alpha/2}\I_{\{x\in D_{2\e}\}}\\
&\qquad\quad +
\e^{-1-\beta+\alpha/2}\left(\I_{\{x+z\in D\backslash D_{2\e}\}}+\I_{\{x\in D\backslash D_{2\e}\}}\right)\Big].
\end{split}
\end{equation}
Here we have also used the fact that
for every $|z|\ge 2\e$, $x+z\in D_{2\e}$ and $x\in D\backslash D_{2\e}$
\begin{align*}
\quad \left|\nabla f_\e(x+z)-\nabla f_\e(x)\right|&\le
|\nabla f_\e(x+z)|+|\nabla f_\e(x)|\\
&\le c_{18}\left(\delta_D(x+z)^{-1+\alpha/2}\I_{\{x+z\in D_{\e}\}}+\delta_D(x)^{-1+\alpha/2}\I_{\{x\in D_{\e}\}}\right)\cdot \left(\frac{|z|}{2\e}\right)^{\beta}\\
&\le c_{19}\e^{-1-\beta+\alpha/2}|z|^\beta.
\end{align*}

Now we write
\begin{align*}
|\LL_{\e} \nabla f_\e(x)| \le  & \, \int_{\{|z|\le 2\e\}}|\nabla f_\e(x+z)-\nabla f_\e(x)|\frac{K\left(x/\e,(x+z)/{\e}\right)}{|z|^{d+\alpha}}\,dz  \\
&  \, +  \int_{\{|z|>2\e\}}|\nabla f_\e(x+z)-\nabla f_\e(x)|\frac{K\left(x/\e,(x+z)/{\e}\right)}{|z|^{d+\alpha}}\,dz.\\
 =:& \, L_1^\e(x)+L_2^\e(x).
\end{align*}

It follows from \eqref{l2-2-3} that
\begin{align*}
L_1^\e(x)&\le c_{20}\left(\sup_{y:|y-x|\le 2\e}|\nabla^2 f_\e(y)|\right)\cdot\int_{\{|z|\le 2\e\}}\frac{|z|}{|z|^{d+\alpha}}\,dz\\
&\le c_{21}\e^{1-\alpha}\left(\delta_D(x)^{-2+\alpha/2}\I_{D_{4\e}}(x)+\e^{-2+\alpha/2}\I_{D\backslash D_{4\e}}(x)\right).
\end{align*}
On the other hand, according to \eqref{l2-3-3},
\begin{align*}
L_2^\e(x)&\le c_{22}\int_{\{|z|>2\e\}}\Big(\delta_D(x+z)^{-1-\beta+\alpha/2}\I_{\{x+z\in D_{3\e}\}}+
\e^{-1-\beta+\alpha/2}\I_{\{x+z\in D_{3\e}\}}\Big)\frac{1}{|z|^{d+\alpha-\beta}}\,dz\\
&\quad +c_{22}\left(\delta_D(x)^{-1-\beta+\alpha/2}\I_{D_{4\e}}(x)+\e^{-1-\beta+\alpha/2}\I_{D\backslash D_{4\e}}(x)\right)
\cdot\int_{\{|z|>2\e\}}\frac{1}{|z|^{d+\alpha-\beta}}\,dz\\
&\le c_{23}\left(\e^{-(\beta-\alpha/2)}\delta_D(x)^{-1-(\alpha-\beta)}+\e^{-(\alpha-\beta)}\delta_D(x)^{-1-\beta+\alpha/2}\right)\I_{D_{4\e}}(x)+
c_{23}\e^{-1-\alpha/2}\I_{D\backslash D_{4\e}}(x),
\end{align*} where in the last inequality we used \eqref{l2-5-3a}, \eqref{l2-5-4a}
and the fact $\beta\in (\alpha/2,\alpha)$.

Putting the estimates for
$L_1^\e(x)$ and $L_2^\e(x)$
together, we
obtain \eqref{l2-3-2}.
\end{proof}

\begin{lemma}\label{l2-8}
The following
statements hold.
\begin{itemize}
\item [(1)] For all $\alpha\in (0,2)$ and $f\in C^{3}_{(-\alpha/2)}(D)$, there is a constant $C_1>0$ such that for all $\e\in (0,1)$ and $x\in D$,
\begin{equation}\label{l2-8-1}
\begin{split}
\left|\bar \LL f_\e(x)-\eta_\e(x)\bar \LL f(x)\right|&\le C_1
\left(\e\delta_D(x)^{-1-\alpha/2}\I_{D_{4\e}}(x)+\e^{-\alpha/2}\I_{D\backslash D_{4\e}}(x)\right),
\end{split}
\end{equation}
where $f_\e := f \eta_\e$.

\item [(2)]
When $\alpha\in (1,2)$
and $f\in C^{4}_{(-\alpha/2)}(D)$,
there is a constant $C_2>0$ such that for all $\e\in (0,1)$, $x\in D$ and $k=1,2$,
\begin{equation}\label{l2-8-2}
\begin{split}
\left|\nabla^k \bar \LL f_\e(x)\right|&\le
C_2(\e^{1-\alpha}\delta_D(x)^{-k-1+\alpha/2}+
\e^{-k+\alpha/2}\delta_D(x)^{-\alpha})\I_{D_{4\e}}(x)
+C_2\e^{-k-\alpha/2}\I_{D \backslash D_{4\e}}(x).
\end{split}
\end{equation}
\item [(3)] When $\alpha\in (0,1]$
and $f\in C^{4}_{(-\alpha/2)}(D)$,
there is a constant $C_3>0$ such that for all $\e\in (0,1)$, $x\in D$ and $k=1,2$,
\begin{equation}\label{l2-8-3}
\begin{split}
\left|\nabla^k
\bar \LL f_\e(x)\right|
&\le
C_3\left(
\delta_D(x)^{-k-\alpha/2}
\I_{\{\alpha\in (0,1)\}}+
\delta_D(x)^{-k-\alpha/2}
\left(1+\left|\log \e \right|\right)\I_{\{\alpha=1\}}\right)\I_{D_{4\e}}(x)\\
&\quad +C_3\e^{-k+\alpha/2}
\delta_D(x)^{-\alpha}
\I_{D_{4\e}}(x)+
C_3\e^{-k-\alpha/2}\left(1+|\log \e|\I_{\{\alpha=1\}}\right)
\I_{D\backslash D_{4\e}}(x).
\end{split}
\end{equation}
\end{itemize}
\end{lemma}
\begin{proof}
(1) By some direct computations, we have
\begin{equation}\label{t2-1-5}
\begin{split}
\bar \LL  f_\e(x)&=\eta_\e(x)\bar \LL f (x)+f(x)\bar \LL \eta_\e(x)+
\int_{\R^d}\left(f(x+z)-f(x)\right)\left(\eta_\e(x+z)-\eta_\e(x)\right)\frac{\bar K}{|z|^{d+\alpha}}\,dz\\
&=\eta_\e(x)\bar \LL f (x)+f(x)\bar \LL \eta_\e(x)+
\int_{\R^d}\left(f(x+z)-f(x)\right)\left(\eta_\e(x+z)-\eta_\e(x)\right)\frac{\bar K}{|z|^{d+\alpha}}\,dz\\
&=:\eta_\e(x)\bar \LL f (x)+I_1^\e(x).
\end{split}
\end{equation}
Combining \eqref{l2-6-1}, \eqref{l2-7-1} with the fact that $|f(x)|\le c_1\delta_D(x)^{\alpha/2}$, we can see
\begin{align*}
|I_1^\e(x)|&\le c_{2}\left(\e\delta_D(x)^{-1-\alpha/2}\I_{D_{4\e}}(x)+\e^{-\alpha/2}\I_{D\backslash D_{4\e}}(x)\right).
\end{align*}
With this we can verify \eqref{l2-8-1} immediately.

\smallskip

(2) For simplicity
we only give a proof for \eqref{l2-8-2} when $k=1$. The $k=2$ case
can be shown
 analogously.
Suppose that $\alpha\in (1,2)$. Then, by \eqref{l2-2-3}, the mean-value theorem and
a
change of variable as well as \eqref{l2-5-1}, it holds
\begin{align*}
&\quad |\nabla \bar \LL f_\e(x)|=|\bar \LL (\nabla f_\e)(x)|\\
&\le \int_{\{|z|\le 2\e\}}\left|\nabla f_\e(x+z)-\nabla f_\e(x)-\left\langle \nabla^2 f_\e(x), z\right\rangle\right|\frac{\bar K}{|z|^{d+\alpha}}dz
+\int_{\{|z|>2\e\}}\left|\nabla f_\e(x+z)-\nabla f_\e(x)\right|\frac{\bar K}{|z|^{d+\alpha}}dz\\
&\le c_{3}\sup_{y\in \R^d:|y-x|\le 2\e}|\nabla^3  f_\e(y)|\cdot\int_{\{|z|\le 2\e\}}\frac{|z|^2}{|z|^{d+\alpha}}\,dz
+c_{3}\int_0^1\int_{\{|z|>2\e\}}|\nabla^2 f_\e(x+sz)|\frac{1}{|z|^{d+\alpha-1}}\,dz\,ds\\
&\le
c_{4}\e^{2-\alpha}\delta_D(x)^{-3+\alpha/2}\I_{D_{4\e}}(x)
+c_{4}\e^{-1-\alpha/2}\I_{D \backslash  D_{4\e}}(x)\\
&\quad +c_{4}\int_0^1s^{1-\alpha}\int_{\{|z|>2s\e\}}\left(\delta_D(x+z)^{-2+\alpha/2}\I_{D_{3\e}}(x+z)+\e^{-2+\alpha/2}\I_{D \backslash D_{3\e}}(x+z)\right)\frac{1}{|z|^{d+\alpha-1}}\,dz\,ds\\
&\le c_{5}\left(\e^{1-\alpha}\delta_D(x)^{-2+\alpha/2}+\e^{-1+\alpha/2}\delta_D(x)^{-\alpha}\right)
\I_{D_{4\e}}(x)
+c_{5}\e^{-1-\alpha/2}\I_{D \backslash  D_{4\e}}(x).
\end{align*}
Hence, \eqref{l2-8-2} is established.

\smallskip

(3) Here we only prove \eqref{l2-8-3} for $k=1$ and $\alpha=1$, since other cases can be proved
by
the same way.
By \eqref{l2-2-3} and
a
change of variable, it holds
\begin{align*}
&\quad |\nabla \bar \LL f_\e(x)|=|\bar \LL (\nabla f_\e)(x)|\\
&\le \int_{\{|z|\le 2\e\}}\left|\nabla f_\e(x+z)-\nabla f_\e(x)-\left\langle \nabla^2 f_\e(x), z\right\rangle\right|\frac{\bar K}{|z|^{d+1}}dz
+\int_{\{2\e<|z|\le 1\}}\left|\nabla f_\e(x+z)-\nabla f_\e(x)\right|\frac{\bar K}{|z|^{d+1}}dz\\
&+\int_{\{|z|>1\}}\left(|\nabla f_\e(x+z)|+|\nabla f_\e(x)|\right)\frac{\bar K}{|z|^{d+1}}dz\\
&\le c_{6}\sup_{y\in \R^d:|y-x|\le 2\e}|\nabla^3  f_\e(y)|\cdot\int_{\{|z|\le 2\e\}}\frac{|z|^2}{|z|^{d+1}}\,dz\\
&\quad +c_{6}\int_0^1\int_{\{2\e<|z|\le 1\}}|\nabla^2 f_\e(x+sz)|\frac{1}{|z|^{d}}\,dz\,ds+c_6\int_{\{|z|>1\}}|\nabla f_\e(x+z)|\frac{1}{|z|^{d+1}}\,dz+c_6|\nabla f_\e(x)|\\
&\le c_{7}\e \delta_D(x)^{-5/2}\I_{D_{4\e}}(x)+c_{7}\e^{-3/2}\I_{D \backslash  D_{4\e}}(x)\\
&\quad +c_{7}\int_0^1\int_{\{2s\e<|z|\le 1\}}\left(\delta_D(x+z)^{-3/2}\I_{D_{3\e}}(x+z)+\e^{-3/2}\I_{D \backslash D_{3\e}}(x+z)\right)\frac{1}{|z|^{d}}\,dz\,ds\\
&\quad +c_{7} \int_{\{|z|>1\}}\left(\delta_D(x+z)^{-1/2}\I_{D_{3\e}}(x+z)+\e^{-1/2}\I_{D \backslash D_{3\e}}(x+z)\right)\frac{1}{|z|^{d}}\,dz\\
&\quad+ c_7\left(\delta_D(x)^{-1/2}\I_{D_{3\e}}(x)+\e^{-1/2}\I_{D \backslash D_{3\e}}(x)\right)\\
&\le c_{8}\left(\e^{-1/2}
\delta_D(x)^{-1}
+(1+|\log \e|)\delta_D(x)^{-3/2}\right)
\I_{D_{4\e}}(x)
+c_{8}\e^{-3/2}(1+|\log \e|)\I_{D \backslash  D_{4\e}}(x),
\end{align*}
where the
last
step follows from \eqref{l2-5-3} and (the proof of) \eqref{l2-5-4}.
This completes the proof.
\end{proof}

\section{Quantitative homogenizations }\label{section3}

This section is devoted to
the proofs of the main results of this paper, Theorems  \ref{t2-1} and  \ref{t2-2}.

When $\alpha\in (1,2)$, we define
$$ F(x)= (F^{(1)}(x),\cdots,F^{(d)} ),\quad   F^{(i)} (x)
:= {\rm p.v.} \int_{\R^d}
z_i
\frac{K(x,x+z)}{|z|^{d+\alpha}}\,dz, \quad i=1,\cdots, d,
$$
where $z_i$ is the $i$th coordinate of $z\in \R^d$.
By the symmetry    and the  multivariate 1-periodicity of   $K(x, y)$   on $\R^d\times \R^d$, as well as a change of variable, we have
\begin{align*}
\int_{\T^d}  F^{(i)} (x)   \,dx
&= \int_{\T^d} \left({\rm p.v.} \int_{\R^d}z_i\frac{K(x,x+z)}{|z|^{d+\alpha}}\,dz\right) dx
={\rm p.v.} \int_{\R^d}   z_i \left(\int_{\T^d}    \frac{K(x+z,x)}{|z|^{d+\alpha}}\,dx\right) dz \\
&=  {\rm p.v.}\int_{\R^d} z_i   \left(\int_{\T^d} \frac{K(x ,x-z)}{|z|^{d+\alpha}}\,dx\right) dz
= - \int_{\T^d} \left( {\rm p.v.} \int_{\R^d}z_i\frac{K(x,x+z)}{|z|^{d+\alpha}}\,dz\right) dx\\
&= -\int_{\T^d}  F^{(i)} (x)   \,dx .
\end{align*}
Thus $\displaystyle\int_{\T^d}  F^{(i)} (x)   \,dx=0$. On the other hand, since $K\in
C^2(\T^d\times \T^d)$, for
every $x\in \T^d$ and $z\in \R^d$ it holds that
\begin{align*}
 \left|z\frac{K(x,x+z)-K(x,x-z)}{|z|^{d+\alpha}}\right|
&\le
2\left(\sup_{x,y\in \T^d}|\nabla_y K(x,y)| \right)\frac{1}{|z|^{d+\alpha-2}}\I_{\{|z|\le 1\}}
\\&\quad+\left(\sup_{x,y\in \T^d}|K(x,y)|\right)\frac{1}{|z|^{d+\alpha-1}}\I_{\{|z|>1\}}.
\end{align*}
This implies that
\begin{align*}
F(x)=\lim_{\e\downarrow 0}\frac{1}{2}\int_{\{|z|>\e\}}z \, \frac{K(x,x+z)-K(x,x-z)}{|z|^{d+\alpha}}\,dz
\end{align*}
is well defined, and
that $F\in C(\T^d)$ by the dominated convergence theorem.

Similarly, for $\alpha\in (0,1]$, define
\begin{align*}
 F_\e(x)= (F_\e^{(1)}(x), \cdots, F_\e^{(d)} ),\quad  F_\e^{(i)} (x)
:= {\rm p.v.} \int_{\{|z|\le {1}/{\e}\}}z_i\frac{K(x,x+z)}{|z|^{d+\alpha}}\,dz, \quad i=1,\cdots, d.
\end{align*}
We can verify
in a similar way as above that,
for all $\alpha\in (0,1]$, $\displaystyle\int_{\T^d}  F_\e^{(i)} (x)   \,dx=0$ and  $F_\e\in C(\T^d)$ so that $$\|F_\e\|_\infty\le c_1( 1+\e^{\alpha-1}\I_{\{\alpha\in (0,1)\}}+|\log \e| \I_{\{\alpha=1\}}).$$
By
Lemma \ref{l2-4}, there exist $\psi\in \D_{\T^d}(\LL)$,
 $\phi =(\phi^{(1)}, \dots, \phi^{(d)})\in C(\T^d;\R^d)$
  with $\phi^{(i)}\in \D_{\T^d}(\LL)$, and
 $\phi_\e =(\phi_\e^{(1)}, \dots, \phi_\e^{(d)})$  with $\phi_\e^{(i)}\in \D_{\T^d}(\LL)$ so that
\begin{equation}\label{t2-1-3}
\begin{split}
 \sL\psi(x) =-(\bar K(x)-\bar K)  &\qquad  \hbox{if } \alpha\in (0,2),\\
 \sL\phi(x) =-F(x)  &\qquad  \hbox{if } \alpha\in (1,2),\\
 \sL\phi_\e(x) =-F_\e(x) &\qquad  \hbox{if } \alpha\in (0,1],
\end{split}
\end{equation}
which satisfy that
\begin{align*}
\|\psi\|_\infty+\|\phi\|_\infty+ \|\nabla \psi\|_\infty+\|\nabla \phi\|_\infty\le c_2 &\quad \hbox{ if } \alpha\in (1,2),\\
\|\psi\|_\infty\le c_2&\quad\hbox{ if }  \alpha\in (0,1]
\end{align*}
and
\begin{equation}\label{t2-1-3a}
\|\phi_\e\|_\infty \le c_2\left(1+\e^{-1+\alpha}\I_{\{\alpha\in (0,1)\}}+ |\log \e| \I_{\{\alpha=1\}}\right).
\end{equation}
Moreover, when $\alpha=1$,
there is
a constant $\theta\in (0,1)$ so that for all $x,y\in \R^d$,
\begin{equation}\label{t2-1-4}
\begin{split}
|\psi(x)-\psi(y)|\le& c_3|x-y|^\theta,\\
|\phi_\e(x)-\phi_\e(y)|\le &c_3\left(1+|\log \e|\right)|x-y|^\theta.
\end{split}
\end{equation}

\medskip

\begin{proof}[Proof of Theorem $\ref{t2-1}$]
Throughout the proof, without
any further mention,
all the constants $c_i$ are independent of $\e$. Recall that, by assumption,
$\bar u\in C^4_{(-\alpha/2)}(D)$.
We divide the proof into three cases.

{\bf Case 1: $\alpha\in (1,2).$}\quad Define
\begin{equation}\label{e:3.4}
v_\e(x):=\bar u_\e(x)+\e\left\langle \phi\left(x/\e\right),\nabla\bar u_\e(x)\right\rangle
+\e^\alpha \bar K^{-1}\psi\left(x/\e\right)\eta_\e(x)\bar \LL\bar u_\e(x),\quad x\in D.
\end{equation}
Here, $\bar u_\e(x):=\bar u(x)\eta_\e(x)$ with $\eta_\e$ being the cut-off function satisfying \eqref{e2-2}--\eqref{e2-1}, and
$\phi$ and $\psi$ are those functions given in \eqref{t2-1-3}.

By the definition \eqref{e:1.6}, \eqref{l2-2-2} and \eqref{l2-2-2a} for $\sL_\e$,$\hat \LL_{1,\e}$ and $\bar \LL_{\e}$ respectively,
\begin{align*}
\LL_\e \bar u_\e(x)&=\hat \LL_{1,\e} \bar u_\e(x)
+\e^{1-\alpha}\left\langle F\left(x/\e\right), \nabla \bar u_\e(x)\right\rangle=\bar \LL_{\e} \bar u_\e(x)+\e^{1-\alpha}\left\langle F\left(x/\e\right), \nabla \bar u_\e(x)\right\rangle+I_1^\e(x),
\end{align*}
where by \eqref{l2-2-1},
$I_1^\e$ satisfies that
\begin{align*}
|I_1^\e(x)|\le  c_1\left(\e^{2-\alpha}\delta_D(x)^{-2+\alpha/2}+\e^{\alpha/2}\delta_D(x)^{-\alpha}\right)\I_{D_{4\e}}(x)
+c_1\e^{-\alpha/2}\I_{D\backslash D_{4\e}}(x).
\end{align*}
In particular,
\begin{align*}
\int_D |I_1^\e(x)|\,dx\le c_2\e^{1-\alpha/2}.
\end{align*}

It holds that
\begin{align*}
 &\quad \LL_\e\left(\e\left\langle \phi\left(\eps^{-1} \cdot \right),\nabla \bar u_\e(\cdot)\right\rangle \right)(x) \\
&=\e\left\langle \LL_\e\phi\left(\eps^{-1} \cdot \right)(x),\nabla \bar u_\e(x)\right\rangle+\e
\left\langle\phi\left(x/\e\right), \LL_\e\left(\nabla \bar u_\e \right)(x) \right\rangle\\
&\quad + \e\int_{\R^d}\left\langle \phi\left((x+z)/{\e}\right)-
\phi\left(x/\e\right), \nabla\bar u_\e(x+z)-\nabla \bar u_\e(x)\right\rangle\frac{K\left(x/\e,(x+z)/{\e}\right)}{|z|^{d+\alpha}}\,dz\\
&=-\e^{1-\alpha}\left\langle F\left(x/\e\right), \nabla \bar u_\e(x)\right\rangle+
\e
\left\langle\phi\left(x/\e\right), \LL_\e\left(\nabla \bar u_\e \right)(x) \right\rangle\\
&\quad + \e\int_{\R^d}\left\langle \phi\left((x+z)/{\e}\right)-
\phi\left(x/\e\right), \nabla\bar u_\e(x+z)-\nabla \bar u_\e(x)\right\rangle\frac{K\left(x/\e,(x+z)/{\e}\right)}{|z|^{d+\alpha}}\,dz\\
&=:-\e^{1-\alpha}\left\langle F\left(x/\e\right), \nabla \bar u_\e(x)\right\rangle+I_{2}^\e(x)+I_{3}^\e(x),
\end{align*}
where in the second equality we used the equation \eqref{t2-1-3}.

Since $\bar u \in C^{3}_{(-\alpha/2)}(D)$, by \eqref{l2-2-1a} and \eqref{l2-2-3},
\begin{align*}
|I_{2}^\e(x)|&=\left|\left\langle \e\phi\left(x/\e\right), \hat \LL_{1,\e}(\nabla \bar u_\e)(x)+\e^{1-\alpha}\nabla^2 \bar u_\e(x)\cdot F\left(x/\e\right)\right\rangle\right|\\
&\le \e\|\phi\|_\infty\left(\left|\hat \LL_{1,\e} \nabla\bar u_\e(x)\right|+\e^{1-\alpha}\|F\|_\infty |\nabla^2 \bar u_\e(x)|\right)\\
&\le c_3\left(\e^{2-\alpha}\delta_D(x)^{-2+\alpha/2}+\e^{\alpha/2}\delta_D(x)^{-\alpha}\right)\I_{D_{4\e}}(x)
+c_3\e^{-\alpha/2}\I_{D \backslash D_{4\e}}(x).
\end{align*}
On the other hand, we write
\begin{align*}
I_{3}^\e(x)&=\e^{1-\alpha}\int_{\{|z|\le 2\}}
\left\langle \phi\left(x/\e+z\right)-
\phi\left(x/\e\right), \nabla\bar u_\e(x+\e z)-\nabla \bar u_\e(x)\right\rangle\frac{K\left(x/\e,x/\e+z\right)}{|z|^{d+\alpha}}\,dz\\
&\quad+\e^{1-\alpha}\int_{\{|z|> 2\}}
\left\langle \phi\left(x/\e+z\right)-
\phi\left(x/\e\right), \nabla\bar u_\e(x+\e z)-\nabla \bar u_\e(x)\right\rangle\frac{K\left(x/\e,x/\e+z\right)}{|z|^{d+\alpha}}\,dz\\
&=:I_{31}^\e(x)+I_{32}^\e(x).
\end{align*}
By the mean-value theorem and \eqref{l2-2-3},
\begin{align*}
|I_{31}^\e(x)|&\le c_4\e^{2-\alpha}\sup_{y\in \R^d:|y-x|\le 2\e}|\nabla^2 \bar u_\e(y)|\cdot \|\nabla \phi\|_\infty\cdot
\int_{\{|z|\le 2\}}\frac{|z|^2}{|z|^{d+\alpha}}\,dz\\
&\le c_5\e^{2-\alpha}\delta_D(x)^{-2+\alpha/2}\I_{D_{4\e}}(x)
+c_5\e^{-\alpha/2}\I_{D \backslash D_{4\e}}(x).
\end{align*}
Applying the mean-value theorem again as well as
a
change of variable, we have
\begin{align*}
|I_{32}^\e(x)|&\le c_6\e\|\phi\|_\infty\cdot\int_{\{|z|>2\e\}}\left|\nabla \bar u_\e(x+z)-\nabla \bar u_\e(x)\right|\frac{1}{|z|^{d+\alpha}}\,dz\\
&\le c_7\e\int_0^1\int_{\{|z|>2\e\}}\left|\nabla^2 \bar u_\e(x+sz)\right|\frac{1}{|z|^{d+\alpha-1}}\,dz\,ds\\
&\le c_8\e\int_0^1\int_{\{|z|>2\e\}}\left(\delta_D(x+sz)^{-2+\alpha/2}\I_{\{x+sz\in D_{3\e}\}}+\e^{-2+\alpha/2}\I_{\{x+sz\in D\backslash D_{3\e}\}}\right)\frac{1}{|z|^{d+\alpha-1}}\,dz\,ds\\
&\le c_9\e\int_0^1s^{1-\alpha}\int_{\{|z|>2s\e\}}\left(\delta_D(x+z)^{-2+\alpha/2}\I_{\{x+z\in D_{3\e}\}}+\e^{-2+\alpha/2}\I_{\{x+z\in D\backslash D_{3\e}\}}\right)
\frac{1}{|z|^{d+\alpha-1}}\,dz\,ds\\
&\le c_{10} \e^{2-\alpha}\delta_D(x)^{-2+\alpha/2}\I_{D_{4\e}}(x)
+c_{10}\e^{\alpha/2}\delta_D(x)^{-\alpha}\I_{D_{4\e}}(x)
+c_{10}\e^{-\alpha/2}\I_{D \backslash D_{4\e}}(x),
\end{align*} where in the last inequality we used \eqref{l2-5-1} and \eqref{l2-5-2}.
Thus, by
all the
estimates above for $I_2^\e(x)$, $I_{31}^\e(x)$ and $I_{32}^\e(x)$, we get
\begin{align*}
\int_D \left(|I_2^\e(x)|+|I_3^\e(x)|\right)\,dx\le
\int_D \left(|I_2^\e(x)|+|I_{31}^\e(x)|+|I_{32}^\e(x)|\right)\,dx\le c_{11}\e^{1-\alpha/2}.
\end{align*}
Moreover, we set
\begin{align*}
&\LL_\e\left(\e^\alpha\psi\left(\eps^{-1} \cdot \right)\bar K^{-1}\eta_\e(\cdot)\bar \LL \bar u_\e(\cdot)\right)(x)\\
&=\e^\alpha\LL_\e\psi\left(\eps^{-1} \cdot \right)(x)\cdot\bar K^{-1}\eta_\e(x)\bar \LL\bar u_\e(x)+
\e^\alpha\bar K^{-1}\psi\left(x/\e\right)\LL_\e (\eta_\e \bar \LL \bar u_\e)(x)\\
&\quad +\e^\alpha \bar K^{-1}
\int_{\R^d}\left(\psi\left((x+z)/{\e}\right)-\psi\left(x/\e\right)\right)
\left(\eta_\e(x+z)\bar \LL \bar u_\e(x+z)-\eta_\e(x)\bar \LL \bar u_\e(x)\right)
\frac{K\left(\frac{x}{\e},\frac{x+z}{\e}\right)}{|z|^{d+\alpha}}\,dz
\\
&=:\eta_\e(x)\left(\bar \LL \bar u_\e(x)-\bar \LL_{\e} \bar u_\e(x)\right)+I_4^\e(x)+I_5^\e(x),
\end{align*}
where in the last equality we have used  the fact
\begin{align}\label{t2-1-5}
\e^\alpha\LL_\e\psi\left(\eps^{-1} \cdot \right)(x)\cdot\bar K^{-1}\bar \LL\bar u_\e(x)=\left(-\bar K\left(x/\e\right)+\bar K\right)\bar K^{-1}\bar \LL u_\e(x)=
\bar \LL \bar u_\e(x)-\bar \LL_{\e} \bar u_\e(x)
\end{align}
that can be verified directly by \eqref{t2-1-3}.
We obtain by \eqref{l2-8-1} and \eqref{e1-4} that
\begin{equation}\label{e:3.5}
 \begin{split}
\bar \LL \bar u_\e(x)
=h(x)\eta_\e(x)+I_6^\e(x),
\end{split}
\end{equation}
where
\begin{align*}
|I_6^\e(x)|&\le c_{12}\left(\e\delta_D(x)^{-1-\alpha/2}\I_{D_{4\e}}(x)+\e^{-\alpha/2}\I_{D\backslash D_{4\e}}(x)\right).
\end{align*}
This yields  that
\begin{equation}\label{t2-1-7a}
|\bar \LL_{\e}\bar u_\e(x)|=\left| {\bar K\left(x/\e\right)}\cdot {\bar K}^{-1}\bar \LL \bar u_\e(x)\right|
\le c_{13}\e^{-\alpha/2}, \quad x\in D\backslash D_{4\e}.
\end{equation}
Using \eqref{l2-8-1} and \eqref{l2-8-2}
(by taking $f=\bar u$ to get
regularity estimates for $\bar \LL \bar u_\e$),
based on these regularity estimates for $\bar \LL \bar u_\e$
and following the arguments for $I_2(x)$, $I_{31}^\e(x)$ and $I_{32}^\e(x)$, we can show that
\begin{align*}
|I_4^\e(x)|+|I_5^\e(x)|&\le c_{14}\left(\e^{2-\alpha}\delta_D(x)^{-2+\alpha/2}+\e^{\alpha/2}\delta_D(x)^{-\alpha}\right)\I_{D_{4\e}}(x)
+c_{14}\e^{-\alpha/2}\I_{D \backslash  D_{4\e}}(x),
\end{align*}
and so
\begin{align*}
\int_D \left(|I_4^\e(x)|+|I_5^\e(x)|\right)\,dx\le c_{15}\e^{1-\alpha/2}.
\end{align*}
 Therefore,
putting
all the estimates for $I_i^\e(x)$, $i=1,\cdots, 5$ above together,
yields that
\begin{equation}\label{t2-1-7}
\LL_\e v_\e(x)=\eta_\e(x)\bar \LL \bar u_\e (x)+(1-\eta_\e(x))\bar \LL_{\e}\bar u_\e(x) +N_1^\e(x),
\end{equation}
where $N_1^\e(x)$ satisfies
\begin{align*}
\int_D |N_1^\e(x)|\,dx\le c_{16}\e^{1-\alpha/2}.
\end{align*}
This, along with \eqref{e1-3} and
\eqref{e:3.5},
in turn gives us that
\begin{align*}
\LL_\e\left(u_\e-v_\e\right)(x)=h(x)(1-\eta_\e(x)^2)-(1-\eta_\e(x))\bar \LL_{\e}\bar u_\e(x)+N_2^\e(x),\quad  x\in D,
\end{align*}
where
\begin{equation}\label{t2-1-8}
\begin{split}
&  \int_D\left(|h(x)(1-\eta_\e(x)^2)|+|(1-\eta_\e(x))\bar \LL_{\e}\bar u_\e(x)|+
|N_2^\e(x)|\right)dx\\
&\le
\int_{D \backslash D_{2\e}}\left(|h(x)|+|\bar \LL_{\e}\bar u_\e(x)|\right)\,dx+\int_D |N_2^\e(x)|\,dx\le c_{17}\e^{1-\alpha/2}.
\end{split}
\end{equation}
Here we have also used \eqref{t2-1-7a} and \eqref{l2-2-1} in the last inequality.

Note that $\left\langle \phi\left(\eps^{-1} \cdot \right), \nabla \bar u_\e(\cdot)\right\rangle\in
{\rm Dom}  (\LL_\e^D)$ and $\psi\left(\eps^{-1} \cdot \right)\eta_\e(\cdot)\LL \bar u_\e(\cdot)\in
{\rm Dom}  (\LL_\e^D)$, so
$u_\e-v_\e\in
{\rm Dom}  (\LL_\e^D)$. Then, according to \eqref{l2-1-1} and \eqref{t2-1-8}, we arrive at
\begin{align*}
\int_D|u_\e(x)-v_\e(x)|\,dx
&\le
c_{19}\int_D\left(|h(x)(1-\eta_\e(x)^2)|+|(1-\eta_\e(x))\bar \LL_{\e}\bar u_\e(x)|+|N_2^\e(x)|\right)\,dx\le c_{20}\e^{1-\alpha/2}.
\end{align*}
On the other hand,
we have by \eqref{e:3.4},
\eqref{e:3.5}  and  \eqref{l2-2-3} that
\begin{align*}
\int_{D}|v_\e(x)-\bar u(x)|\,dx
&\le \int_D |\bar u(x)||1-\eta_\e(x)|\,dx+
\e\|\phi\|_\infty \int_D |\nabla \bar u_\e(x)|\,dx+
\e^\alpha \|\psi\|_\infty \bar K^{-1}\int_D |\bar\LL \bar u_\e(x)|\,dx\\
&\le c_{21}\e,
\end{align*} where in the last inequality we used the fact that $D$ is bounded.
Therefore,
\begin{align*}
\int_D|u_\e(x)-\bar u(x)|\,dx\le \int_D|u_\e(x)-v_\e(x)|\,dx+\int_{D}|v_\e(x)-\bar u(x)|\,dx\le c_{22}\e^{1-\alpha/2}.
\end{align*} The proof for $\alpha\in (1,2)$ is complete.

{\bf Case 2: $\alpha=1$.}\quad In this case, we set
$$
v_\e(x):=\bar u_\e(x)+\e\left\langle \phi_\e\left(x/\e\right),\nabla\bar u_\e(x)\right\rangle
+\e\bar K^{-1}\psi\left(x/\e\right)\eta_\e(x)\bar \LL\bar u_\e(x),\quad x\in D,
$$
where $\phi_\e$ and $\psi$ are those functions given in \eqref{t2-1-3}.

According to \eqref{l2-3-1}, we obtain
\begin{align*}
\LL_\e \bar u_\e(x)&=\hat \LL_{2,\e} \bar u_\e(x)
+\left\langle F_\e\left(x/\e\right), \nabla \bar u_\e(x)\right\rangle=\bar \LL_{\e} \bar u_\e(x)+\left\langle F_\e\left(x/\e\right), \nabla \bar u_\e(x)\right\rangle+J_1^\e(x),
\end{align*}
where $\hat \LL_{2,\e}$ is defined by \eqref{l2-2-2}, and
$J_1^\e(x)$ satisfies
\begin{align*}
|J_1^\e(x)|\le  c_{1}\left(\e^{1/2}
\delta_D(x)^{-1}
+\e\delta_D(x)^{-3/2}(1+|\log \e|)\right)\I_{D_{4\e}}(x)
+c_{1}\e^{-1/2}(1+|\log \e|)\I_{D\backslash D_{4\e}}(x),
\end{align*}
which in particular implies that
\begin{align*}
\int_D |J_1^\e(x)|\,dx\le c_2\e^{1/2}(1+|\log \e|).
\end{align*}

As in Case 1,
by \eqref{t2-1-3},
\begin{align*}
 &  \LL_\e\left(\e\left\langle \phi_\e\left(\eps^{-1} \cdot \right),\nabla \bar u_\e(\cdot)\right\rangle \right)(x) \\
&=-\left\langle F_\e\left(x/\e\right), \nabla \bar u_\e(x)\right\rangle+
\e\left\langle\phi_\e\left(x/\e\right), \LL_\e\left(\nabla \bar u_\e \right)(x) \right\rangle\\
&\quad + \e\int_{\R^d}\left\langle \phi_\e\left((x+z)/{\e}\right)-
\phi_\e\left(x/\e\right), \nabla\bar u_\e(x+z)-\nabla \bar u_\e(x)\right\rangle\frac{K\left(x/\e,(x+z)/{\e}\right)}{|z|^{d+1}}\,dz\\
&=:-\left\langle F_\e\left(x/\e\right), \nabla \bar u_\e(x)\right\rangle+J_{2}^\e(x)+J_{3}^\e(x).
\end{align*}
Using \eqref{l2-2-3}, \eqref{l2-3-1a} and \eqref{t2-1-3a}, we obtain
\begin{align*}
|J_2^\e(x)|&\le c_3 \e\|\phi_\e\|_\infty \left(|\hat \LL_{2,\e}(\nabla\bar u_\e)(x)|+\|F_\e\|_\infty|\nabla^2 \bar u_\e(x)|\right)\\
&\le c_4\e(1+|\log \e|)\left(\delta_D(x)^{-3/2}(1+|\log \e|)+\e^{-1/2}
\delta_D(x)^{-1}
\right)\I_{D_{4\e}}(x)\\
&\quad +c_4\e^{-1/2}(1+|\log \e|^2)\I_{D\backslash D_{4\e}}(x).
\end{align*}
Set
\begin{align*}
J_{3}^\e(x)&=\int_{\{|z|\le 2\}}
\left\langle \phi_\e\left(x/\e+z\right)-
\phi_\e\left(x/\e\right), \nabla\bar u_\e(x+\e z)-\nabla \bar u_\e(x)\right\rangle\frac{K\left(x/\e,x/\e+z\right)}
{|z|^{d+1}}\,dz\\
&\quad+\int_{\{|z|> 2\}}
\left\langle \phi_\e\left(x/\e+z\right)-
\phi_\e\left(x/\e\right), \nabla\bar u_\e(x+\e z)-\nabla \bar u_\e(x)\right\rangle\frac{K\left(x/\e,x/\e+z\right)}{|z|^{d+1}}\,dz\\
&=:J_{31}^\e(x)+J_{32}^\e(x).
\end{align*}
By the mean-value theorem and \eqref{t2-1-4}, we have
\begin{align*}
|J_{31}^\e(x)|&\le c_5\e(1+|\log \e|)\sup_{y\in \R^d:|y-x|\le 2\e}|\nabla^2 \bar u_\e(y)|\cdot
\int_{\{|z|\le 2\}}
\frac{|z|^{1+\theta}}{|z|^{d+1}}
\,dz\\
&\le c_6\e(1+|\log \e|)\delta_D(x)^{-3/2}\I_{D_{4\e}}(x)
+c_{6}\e^{-1/2}(1+|\log \e|)\I_{D \backslash D_{4\e}}(x).
\end{align*}
Applying the mean-value theorem again, \eqref{t2-1-3a} and
a
change of variable, it holds that
\begin{align*}
&|J_{32}^\e(x)|\\
&\le c_7\e\|\phi_\e\|_\infty\cdot
\left(\int_{\{2\e<|z|\le 1\}}\left|\nabla \bar u_\e(x+z)-\nabla \bar u_\e(x)\right|\frac{1}{|z|^{d+1}}\,dz
+\int_{\{|z|>1\}}\left(|\nabla \bar u_\e(x+z)|+|\nabla \bar u_\e(x)|\right)\frac{1}{|z|^{d+1}}\,dz\right)\\
&\le c_8\e(1+|\log \e|)\bigg[\int_0^1\int_{\{2\e<|z|\le 1\}}\left|\nabla^2 \bar u_\e(x+sz)\right|\frac{1}{|z|^{d}}\,dz\,ds+
\left(\delta_D(x)^{-1/2}\I_{D_{4\e}}(x)+\e^{-1/2}\I_{D\backslash D_{4\e}}(x)\right)\\
&\qquad\qquad\qquad \qquad\,\, +\int_{\{|z|>1\}} \left(\delta_D(x+z)^{-1/2}\I_{\{x+z\in D_{4\e}\}}+\e^{-1/2}\I_{\{x+z\in D\backslash D_{4\e}\}}\right)\frac{1}{|z|^{d+1}}\,dz\bigg]\\
&\le c_{9}\e(1+|\log \e|)\bigg[\int_0^1\int_{\{2s\e<|z|\le s\}}\left(\delta_D(x+z)^{-3/2}\I_{\{x+z\in D_{3\e}\}}+\e^{-3/2}\I_{\{x+z\in D\backslash D_{3\e}\}}\right)
\frac{1}{|z|^{d}}\,dz\,ds\\
&\qquad\qquad\qquad \qquad \,\, +\left(\delta_D(x)^{-1/2}\I_{D_{4\e}}(x)+\e^{-1/2}\I_{D\backslash D_{4\e}}(x)\right)\\
&\qquad\qquad\qquad \qquad \,\,+\int_{\{|z|>1\}} \left(\delta_D(x+z)^{-1/2}\I_{\{x+z\in D_{4\e}\}}+\e^{-1/2}\I_{\{x+z\in D\backslash D_{4\e}\}}\right)\frac{1}{|z|^{d+1}}\,dz\bigg].
\end{align*}
Hence, by \eqref{l2-5-3}--\eqref{l2-5-5} and all the estimates above for $J_2^\e(x)$, $J_{31}^\e(x)$ and $J_{32}^\e(x)$,  we get
\begin{align*}
\int_D \left(|J_2^\e(x)|+|J_3^\e(x)|\right)\,dx\le
\int_D \left(|J_2^\e(x)|+|J_{31}^\e(x)|+|J_{32}^\e(x)|\right)\,dx\le c_{10}\e^{1/2}(1+|\log \e|^2).
\end{align*}

Next, we define
\begin{align*}
&\LL_\e\left(\e\psi\left(\eps^{-1} \cdot \right)\bar K^{-1}\eta_\e(\cdot)\bar \LL \bar u_\e(\cdot)\right)(x)\\
&=\e \LL_\e\psi\left(\eps^{-1} \cdot \right)(x)\cdot\bar K^{-1}\eta_\e(x)\bar \LL\bar u_\e(x)+
\e \bar K^{-1}\psi\left(x/\e\right)\LL_\e (\eta_\e \bar \LL \bar u_\e)(x)\\
&\quad +\e \bar K^{-1}
\int_{\R^d}\left(\psi\left((x+z)/{\e}\right)-\psi\left(x/\e\right)\right)
\left(\eta_\e(x+z)\bar \LL \bar u_\e(x+z)-\eta_\e(x)\bar \LL \bar u_\e(x)\right)
\frac{K\left(x/\e,(x+z)/\e\right)}{|z|^{d+1}}\,dz\\
&=:\eta_\e(x)\left(\bar \LL \bar u_\e(x)-\bar \LL_{\e} \bar u_\e(x)\right)+J_4^\e(x)+J_5^\e(x).
\end{align*}
Here we have used the fact \eqref{t2-1-5}.

By \eqref{l2-8-1} and  \eqref{e1-3}, we know
\begin{equation}\label{t2-1-9}
\begin{split}
\bar \LL \bar u_\e(x)=h(x)\eta_\e(x)+J_6^\e(x).
\end{split}
\end{equation}
Here
\begin{align*}
|J_6^\e(x)|&\le c_{11}\left(\e\delta_D(x)^{-3/2}\I_{D_{4\e}}(x)+\e^{-1/2}\I_{D\backslash D_{4\e}}(x)\right).
\end{align*}

Using \eqref{l2-8-1} and \eqref{l2-8-3} (to get regularity estimates for $\bar \LL \bar u_\e$),
based on these regularity estimates for $\bar \LL \bar u_\e$
and following the arguments for
$J_2^\e(x)$,
$J_{31}^\e(x)$ and $J_{32}^\e(x)$, we can show that
\begin{align*}
\int_D \left(|J_4^\e(x)|+|J_5^\e(x)|\right)\,dx\le c_{12}\e^{1/2}(1+|\log \e|^2).
\end{align*}

Putting
all the estimates for $J_i^\e(x)$, $i=1,\cdots, 5$, above  together yields that
\begin{equation}
\label{e:3.10}
\LL_\e v_\e(x)=\eta_\e(x)\bar \LL \bar u_\e (x)+(1-\eta_\e(x))\bar \LL_{\e}\bar u_\e(x) +N_3^\e(x),
\end{equation}
where
\begin{align*}
\int_D |N_3^\e(x)|\,dx\le c_{13}\e^{1/2}(1+|\log \e|^2).
\end{align*}
This, along with \eqref{e1-3} and \eqref{t2-1-9}, gives us that
\begin{align*}
\LL_\e\left(u_\e-v_\e\right)(x)=h(x)(1-\eta_\e(x)^2)-(1-\eta_\e(x))\bar \LL_{\e}\bar u_\e(x)+N_4^\e(x),\quad  x\in D,
\end{align*}
where
\begin{equation}\label{t2-1-9a}
\begin{split}
&\quad \int_D\left(|h(x)(1-\eta_\e(x)^2)|+|(1-\eta_\e(x))\bar \LL_{\e}\bar u_\e(x)|+
|N_4^\e(x)|\right)\,dx\le c_{14}\e^{1/2}(1+|\log \e|^2).
\end{split}
\end{equation}
Here we have also used \eqref{t2-1-7a} (which still holds for $\alpha\in (0,1]$ by its proof).
As explained in the proof
of Case 1,
we can apply \eqref{t2-1-9a} and \eqref{l2-1-1} to obtain
\begin{align*}
\int_D |u_\e(x)-v_\e(x)|\,dx\le c_{15}\e^{1/2}(1+|\log \e|^2).
\end{align*}

Combining
\eqref{e:3.4} and \eqref{t2-1-9} with \eqref{l2-2-3} yields  that
\begin{align*}
&   \int_{D}|v_\e(x)-\bar u(x)|\,dx\\
&\le \int_D |\bar u(x)||1-\eta_\e(x)|\,dx+
\e\|\phi_\e\|_\infty \int_D |\nabla \bar u_\e(x)|\,dx+
\e \|\psi\|_\infty \bar K^{-1}\int_D |\bar\LL \bar u_\e(x)|\,dx\\
&\le c_{16}\e^{1/2}(1+|\log \e|^2).
\end{align*}
Therefore,
\begin{align*}
\int_{D}|u_\e(x)-\bar u(x)|\,dx\le \int_D |u_\e(x)-v_\e(x)|\,dx+\int_{D}|v_\e(x)-\bar u(x)|\,dx\le c_{17}\e^{1/2}(1+|\log \e|^2).
\end{align*}
Hence, the proof for the case $\alpha=1$ is finished.

{\bf Case 3: $\alpha\in (0,1)$.}\quad Define
$$
v_\e(x):=\bar u_\e(x)+\e\left\langle \phi_\e\left(x/\e\right),\nabla\bar u_\e(x)\right\rangle
+\e^\alpha\bar K^{-1}\psi\left(x/\e\right)\eta_\e(x)\bar \LL\bar u_\e(x),\quad x\in D.
$$
Here $\phi_\e$ and $\psi$ are those functions in \eqref{t2-1-3}.
Therefore, by \eqref{l2-3-1}, we obtain
\begin{align*}
\LL_\e \bar u_\e(x)&=\hat \LL_{2,\e} \bar u_\e(x)
+\e^{1-\alpha}\left\langle F_\e\left(x/\e\right), \nabla \bar u_\e(x)\right\rangle=\bar \LL_{\e} \bar u_\e(x)+
\e^{1-\alpha}\left\langle F_\e\left(x/\e\right), \nabla \bar u_\e(x)\right\rangle+L_1^\e(x),
\end{align*}
where
$L_1^\e$ satisfies
\begin{align*}
|L_1^\e(x)|\le  c_{1}\e^{\alpha/2}
\delta_D(x)^{-\alpha}
\I_{D_{4\e}}(x)
+c_{1}\e^{-1+\alpha/2}\I_{D \backslash D_{4\e}}(x)
\end{align*}
that implies
\begin{align*}
\int_D |L_1^\e(x)|\,dx\le c_2\e^{\alpha/2}.
\end{align*}

As before, by \eqref{t2-1-3}, we have
\begin{align*}
 & \LL_\e\left(\e\left\langle \phi_\e\left(\eps^{-1} \cdot \right),\nabla \bar u_\e(\cdot)\right\rangle \right)(x) \\
&=-\e^{1-\alpha}\left\langle F_\e\left(x/\e\right), \nabla \bar u_\e(x)\right\rangle
+\e\left\langle\phi_\e\left(x/\e\right), \LL_\e\left(\nabla \bar u_\e \right)(x) \right\rangle\\
&\quad + \e\int_{\R^d}\left\langle \phi_\e\left((x+z)/{\e}\right)-
\phi_\e\left(x/\e\right), \nabla\bar u_\e(x+z)-\nabla \bar u_\e(x)\right\rangle\frac{K\left(x/\e,(x+z)/{\e}\right)}{|z|^{d+\alpha}}\,dz\\
&=:-\e^{1-\alpha}\left\langle F_\e\left(x/\e\right), \nabla \bar u_\e(x)\right\rangle+L_{2}^\e(x)+L_{3}^\e(x).
\end{align*}
According to \eqref{l2-2-3}, \eqref{l2-3-2} and \eqref{t2-1-3a}, we obtain
that for every $\beta\in (\alpha/2,\alpha)$,
\begin{align*}
|L_2^\e(x)|&\le c_3 \e\|\phi_\e\|_\infty |\LL_{\e}(\nabla\bar u_\e)(x)|\\
&\le c_{4}\e^\alpha\left(\e^{-(\beta-\alpha/2)}\delta_D(x)^{-1-(\alpha-\beta)}+\e^{-(\alpha-\beta)}\delta_D(x)^{-1-\beta+\alpha/2}\right)\I_{D_{4\e}}(x)+
c_{4}\e^{-1+\alpha/2}\I_{D\backslash D_{4\e}}(x).
\end{align*}
We set
\begin{align*}
L_{3}^\e(x)&=\e^{1-\alpha}\int_{\{|z|\le 2\}}
\left\langle \phi_\e\left(x/\e+z\right)-
\phi\left(x/\e\right), \nabla\bar u_\e(x+\e z)-\nabla \bar u_\e(x)\right\rangle\frac{K\left(x/\e,x/\e+z\right)}{|z|^{d+\alpha}}\,dz\\
&\quad+\e^{1-\alpha}\int_{\{|z|> 2\}}
\left\langle \phi_\e\left(x/\e+z\right)-
\phi_\e\left(x/\e\right), \nabla\bar u_\e(x+\e z)-\nabla \bar u_\e(x)\right\rangle\frac{K\left(x/\e,x/\e+z\right)}{|z|^{d+\alpha}}\,dz\\
&=:L_{31}^\e(x)+L_{32}^\e(x).
\end{align*}
By the mean-value theorem and \eqref{t2-1-3a}, we have
\begin{align*}
|L_{31}^\e(x)|&\le c_5\e^{2-\alpha}\|\phi_\e\|_\infty\cdot\sup_{y\in \R^d:|y-x|\le 2\e}|\nabla^2 \bar u_\e(y)|\cdot
\int_{\{|z|\le 2\}}\frac{|z|}{|z|^{d+\alpha}}\,dz\\
&\le c_6\e\delta_D(x)^{-2+\alpha/2}\I_{D_{4\e}}(x)
+c_{6}\e^{-1+\alpha/2}\I_{D \backslash D_{4\e}}(x).
\end{align*}
Applying \eqref{l2-3-3} (with $f=\bar u$ in \eqref{l2-3-3}) and \eqref{t2-1-3a}, we get
\begin{align*}
|L_{32}^\e(x)|&\le c_7\e\|\phi_\e\|_\infty\cdot
\int_{\{|z|>2\e\}}\left|\nabla \bar u_\e(x+z)-\nabla \bar u_\e(x)\right|\frac{1}{|z|^{d+\alpha}}\,dz\\
&\le c_8\e^\alpha\int_{\{|z|>2\e\}}\Big(\delta_D(x+z)^{-1-\beta+\alpha/2}\I_{\{x+z\in D_{3\e}\}}+
\e^{-1-\beta+\alpha/2}
\I_{\{x+z\in D\backslash D_{3\e}\}}
\Big)\frac{1}{|z|^{d+\alpha-\beta}}\,dz\\
&\quad +c_{8}\e^\alpha\left(\delta_D(x)^{-1-\beta+\alpha/2}\I_{D_{4\e}}(x)+\e^{-1-\beta+\alpha/2}\I_{D\backslash D_{4\e}}(x)\right)
\cdot\int_{\{|z|>2\e\}}\frac{1}{|z|^{d+\alpha-\beta}}\,dz\\
&\le c_9\left(\e^{-(\beta- {3\alpha}/{2})}\delta_D(x)^{-1-(\alpha-\beta)}+\e^{\beta}\delta_D(x)^{-1-\beta+\alpha/2}\right)\I_{D_{3\e}}(x)+c_9
\e^{-1+\alpha/2}\I_{D\backslash D_{3\e}}(x),
\end{align*}
where the last inequality follows from \eqref{l2-5-3a} and \eqref{l2-5-4a}.
Hence, by all the estimates above for $L_2^\e(x)$, $L_{31}^\e(x)$ and $L_{32}^\e(x)$
and using the fact $\beta\in (\alpha/2,\alpha)$,
we get
\begin{align*}
\int_D \left(|L_2^\e(x)|+|L_3^\e(x)|\right)\,dx\le
\int_D \left(|L_2^\e(x)|+|L_{31}^\e(x)|+|L_{32}^\e(x)|\right)\,dx\le c_{10}\e^{\alpha/2}.
\end{align*}

Define
\begin{align*}
&  \LL_\e\left(\e^\alpha\psi\left(\eps^{-1} \cdot \right)\bar K^{-1}\eta_\e(\cdot)\bar \LL \bar u_\e(\cdot)\right)(x)\\
&=\e^\alpha \LL_\e\psi\left(
 \e^{-1} \cdot
\right)(x)\cdot\bar K^{-1}\eta_\e(x)\bar \LL\bar u_\e(x)+
\e^\alpha \bar K^{-1}\psi\left(x/\e\right)\LL_\e (\eta_\e \bar \LL \bar u_\e)(x)\\
&\quad +\e^\alpha \bar K^{-1}
\int_{\R^d}\left(\psi\left((x+z)/{\e}\right)-\psi\left(x/\e\right)\right)
\left(\eta_\e(x+z)\bar \LL \bar u_\e(x+z)-\eta_\e(x)\bar \LL \bar u_\e(x)\right)\frac{1}{|z|^{d+\alpha}}\,dz\\
&=:\eta_\e(x)\left(\bar \LL \bar u_\e(x)-\bar \LL_{\e} \bar u_\e(x)\right)+L_4^\e(x)+L_5^\e(x).
\end{align*}
Here we have used the property \eqref{t2-1-5} again.
By \eqref{l2-8-1} and  \eqref{e1-3}, we know
\begin{equation}\label{t2-1-10}
\begin{split}
\bar \LL \bar u_\e(x)=h(x)\eta_\e(x)+L_6^\e(x)
\end{split}
\end{equation}
where
\begin{align*}
|L_6^\e(x)|&\le c_{11}\left(\e\delta_D(x)^{-1-\alpha/2}\I_{D_{4\e}}(x)+\e^{-\alpha/2}\I_{D\backslash D_{4\e}}(x)\right).
\end{align*}

Furthermore, we have
\begin{align*}
|L_5^\e(x)|
&\le c_{12}\e^\alpha\|\psi\|_\infty\left(
\int_{\{|z|\le 2\e\}}+\int_{\{|z|> 2\e\}} \right) \left|\eta_\e(x+z)\bar \LL \bar u_\e(x+z)-\eta_\e(x)\bar \LL \bar u_\e(x)\right|\frac{1}{|z|^{d+\alpha}}\,dz \\
&=:L_{51}^\e(x)+L_{52}^\e(x).
\end{align*}
According to \eqref{l2-8-3},
\begin{align*}
L_{51}^\e(x)&\le c_{13}\e^\alpha\sup_{y:|y-x|\le 2\e}|\nabla (\eta_\e(y)\bar \LL \bar u_\e(y))|\cdot\int_{\{|z|\le 2\e\}}\frac{1}{|z|^{d+\alpha-1}}\,dz \\
&\le c_{14}\e\left(\delta_D(x)^{-1-\alpha/2}
\I_{D_{4\e}}(x)+\e^{-1-\alpha/2}\I_{D\backslash D_{4\e}}(x)\right)\\
&= c_{14}\e\delta_D(x)^{-1-\alpha}\I_{D_{4\e}}(x)+c_{14}\e^{-\alpha/2}\I_{D\backslash D_{4\e}}(x).
\end{align*}

Using \eqref{t2-1-10}, we get
\begin{align*}
L_{52}^\e(x)&\le c_{15}\e^\alpha\int_{\{|z|>2\e\}}\left|h(x+z)\eta^2_\e(x+z)-h(x)\eta^2_\e(x)\right|\frac{1}{|z|^{d+\alpha}}\,dz\\
&\quad +c_{15}\e^\alpha\int_{\{|z|>2\e\}}\left(|L_6^\e(x)|+|L_6^\e(x+z)|\right)\frac{1}{|z|^{d+\alpha}}\,dz\\
&=:L_{521}^\e(x)+L_{522}^\e(x).
\end{align*}
Noting that $h\in C^{1}_{(\alpha/2)}(D)\subset C^\beta_{(\alpha/2)}(D)$ for every $\beta\in (\alpha/2,\alpha)$, by the definition and the argument for \eqref{l2-3-3}, it holds that for all $x,z\in \R^d$ with $|z|>2\e$,
\begin{equation}\label{t2-1-11}
\begin{split}
 \frac{|h (x+z)\eta^2_\e(x+z)-h (x)\eta^2_\e(x)|}{|z|^\beta}
 \le
c_{16}\Big(&\delta_D(x+z)^{-\beta-\alpha/2}\I_{\{x+z\in D_{3\e}\}}+
\delta_D(x)^{-\beta-\alpha/2}\I_{\{x\in D_{3\e}\}}
 \\
 &+
\e^{-\beta-\alpha/2}\left(\I_{\{x+z\in D\backslash D_{3\e}\}}+\I_{\{x\in D\backslash D_{3\e}\}}\right)\Big).
\end{split}
\end{equation}
By \eqref{t2-1-11}, we deduce that
\begin{align*}
L_{521}^\e(x)\le & c_{17}\e^\alpha\int_{\{|z|>2\e\}}\left(\delta_D(x+z)^{-\beta-\alpha/2}\I_{\{x+z\in D_{3\e}\}}+
\e^{-\beta-\alpha/2}\I_{\{x+z\in D\backslash D_{3\e}\}}\right)\frac{1}{|z|^{d+\alpha-\beta}}\,dz\\
&+c_{17}\e^\alpha\left(\delta_D(x)^{-\beta-\alpha/2}\I_{D_{3\e}}(x)+
\e^{-\beta-\alpha/2}\I_{D\backslash D_{3\e}}(x)\right)\cdot \int_{\{|z|>2\e\}}\frac{1}{|z|^{d+\alpha-\beta}}\,dz\\
\le & c_{18}\left(\e^{\beta}\delta_D(x)^{-\beta-\alpha/2}+
(\e^\alpha\vee \e^{1-\beta
+\alpha/2})\cdot
\delta_D(x)^{-1-(\alpha-\beta)}\right)\I_{D_{4\e}}(x)+c_{19}\e^{-\alpha/2}\I_{D\backslash D_{4\e}}(x),
\end{align*}
where in the last inequality we have used \eqref{l2-5-4a} and the fact
\begin{align*}
&\int_{\{|z|>2\e\}}\delta_D(x+z)^{-\beta-\alpha/2}\I_{\{x+z\in D_{3\e}\}}\frac{1}{|z|^{d+\alpha-\beta}}\,dz\\
&\le c_{19}\left(
(\e^{1-\beta-\alpha/2}\vee 1)\cdot
\delta_D(x)^{-1-(\alpha-\beta)}+\e^{-(\alpha-\beta)}\delta_D(x)^{-\beta-\alpha/2}\right)\I_{D_{4\e}}(x)+c_{19}
\e^{-{3\alpha}/{2}}\I_{D\backslash D_{4\e}}(x)
\end{align*}
that can be proved by
following the proof of \eqref{l2-5-3a} line by line.
According to \eqref{t2-1-10}, we obtain
\begin{align*}
L_{522}^\e(x)&\le c_{20}\e^{1+\alpha}\int_{\{|z|>2\e\}}\left(
\delta_D(x+z)^{-1-\alpha/2}
\I_{\{x+z\in D_{3\e}\}}+\e^{-1-\alpha/2}
\I_{\{x+z\in D\backslash D_{3\e}\}}\right)\frac{1}{|z|^{d+\alpha}}\,dz\\
&\quad +c_{20}\left(\e\delta_D(x)^{-1-\alpha/2}\I_{D_{4\e}}(x)+\e^{-\alpha/2}\I_{D\backslash D_{4\e}}(x)\right)\\
&\le c_{21}\left(
\e
\delta_D(x)^{-1-\alpha/2}\I_{D_{4\e}}(x)+\e^{-\alpha/2}\I_{D\backslash D_{4\e}}(x)\right),
\end{align*}
where the last inequality we have used \eqref{l2-5-4} and the fact that
\begin{align*}
&\int_{\{|z|>2\e\}}\delta_D(x+z)^{-1-\alpha/2}\I_{\{x+z\in D_{3\e}\}}\frac{1}{|z|^{d+\alpha}}\,dz\\
&\le c_{22}\left(\e^{-\alpha}\delta_D(x)^{-1-\alpha/2}+\e^{-\alpha/2}\delta_D(x)^{-1-\alpha}\right)\I_{D_{4\e}}(x)+ c_{22}
\e^{-{3\alpha}/{2}}\I_{D\backslash D_{4\e}}(x)
\end{align*}
that can be verified by the same way as that for \eqref{l2-5-3a}.

Using the same procedure
(of $I_4^\e(x)$ and $J_4^\e(x)$)
to estimate $L_4^\e(x)$ and putting all the estimates above together, as well as taking into account that $\beta>\alpha/2$, yield that
\begin{align*}
\int_D \left(|L_4^\e(x)|+|L_5^\e(x)|\right)\,dx\le
c_{23}\e^{1-\alpha/2}\le c_{23}\e^{\alpha/2}.
\end{align*}

Combining
all the estimates for $L_i^\e(x)$, $i=1,\cdots, 5$, above, we have
\begin{equation}
\label{e:3.13}
\LL_\e v_\e(x)=\eta_\e(x)\bar \LL \bar u_\e (x)+(1-\eta_\e(x))\bar \LL_{\e}\bar u_\e(x) +N_5^\e(x),
\end{equation}
where
\begin{align*}
\int_D |N_5^\e(x)|\,dx\le c_{24}\e^{\alpha/2}.
\end{align*}
According to this, \eqref{e1-3} and
\eqref{t2-1-10}
yield that
\begin{align*}
\LL_\e\left(u_\e-v_\e\right)(x)=h(x)(1-\eta_\e(x)^2)-(1-\eta_\e(x))\bar \LL_{\e}\bar u_\e(x)+N_6^\e(x),\quad  x\in D,
\end{align*}
where
\begin{equation}\label{t2-1-12}
\begin{split}
&\int_D\left(|h(x)(1-\eta_\e(x)^2)|+|(1-\eta_\e(x))\bar \LL_{\e}\bar u_\e(x)|+
|N_6^\e(x)|\right)dx\le c_{25}\e^{\alpha/2}.
\end{split}
\end{equation}
Here we have also used \eqref{t2-1-7a} (by the proof which still holds for $\alpha\in (0,1]$).

As explained in the proof
of Case 1,
according to \eqref{t2-1-12}, we can apply \eqref{l2-1-1} to obtain
\begin{align*}
\int_D |u_\e(x)-v_\e(x)|\,dx\le c_{26}\e^{\alpha/2}.
\end{align*}

On the other hand, combining \eqref{l2-2-3} with \eqref{t2-1-10}
gives us
\begin{align*}
  \int_{D}|v_\e(x)-\bar u(x)|\,dx&\le \int_D |\bar u(x)||1-\eta_\e(x)|\,dx+
\e\|\phi_\e\|_\infty \int_D |\nabla \bar u_\e(x)|\,dx+
\e^\alpha \|\psi\|_\infty \bar K^{-1}\int_D |\bar \LL \bar u_\e(x)|\,dx\\
&\le c_{27}\e^{\alpha/2},
\end{align*}
and so
\begin{align*}
\int_{D}|u_\e(x)-\bar u(x)|\,dx\le \int_D |u_\e(x)-v_\e(x)|\,dx+\int_{D}|v_\e(x)-\bar u(x)|\,dx\le c_{28}\e^{\alpha/2}.
\end{align*}
Therefore, we have finished the proof for $\alpha\in (0,1)$.
\end{proof}

Next, we will present the

\begin{proof}[Proof of Theorem $\ref{t2-2}$] The proof is almost the same as (but simpler than) that of
Theorem \ref{t2-1}, so we only give some crucial different points here.
Throughout the proof, all the constants $c_i$ are independent of $\e$.
First,
due to the fact $\bar u \in C_c^2(D)$,
there is a constant $c_1>0$ so that for all $x\in D$ and $k=0,1,2$,
\begin{align}\label{t2-2-2}
|\nabla^k \bar u(x)|\le c_1.
\end{align}
With this at hand
(in particular we do not need to consider blow up behaviors of $|\nabla^k \bar u|$ near $\partial D$), it is easy to verify that there is a constant $c_2>0$ so that for all $x\in D$,
\begin{equation}\label{t2-2-3}
\begin{split}
&|\hat \LL_{1,\e} \bar u(x)|+|\hat \LL_{1,\e} \nabla \bar u(x)|\le c_2  \quad \hbox{ if }  \alpha\in (1,2),\\
&|\hat \LL_{2,\e} \bar u(x) |+|\hat \LL_{2,\e} \nabla \bar u(x)|\le c_2 \quad \hbox{ if }  \alpha\in (0,1],
\end{split}
\end{equation} where $\hat \LL_{1,\e}$ and $\hat \LL_{2,\e}$ are defined by \eqref{l2-2-2}.

Using \eqref{t2-2-2} in place of \eqref{l2-2-3}, and following  the proofs of Lemmas \ref{l2-2} and \ref{l2-3} (which are in fact simpler since
we do not need to consider the blow up behaviours of
$\nabla^k \bar u$
near the boundary), we can prove that there is a constant $c_3>0$ so that for all $x\in D$,
\begin{equation}\label{t2-2-4}
\begin{split}
 |\hat \LL_{1,\e} \bar u(x)-\bar \LL_{\e} \bar u(x)|
&\le c_3\e^{2-\alpha}\quad \hbox{ if } \alpha\in (1,2)
\end{split}
\end{equation} and
\begin{equation}\label{t2-2-5}
 |\hat \LL_{2,\e} \bar u(x)-\bar \LL_{\e} \bar u(x)|
\le
\begin{cases}
c_3\e(1+|\log \e|)\ &\quad \hbox{ if } \alpha=1,\\
c_3\e^{\alpha}\ & \quad \hbox{ if } \alpha\in (0,1),\\
\end{cases}
\end{equation} where $\bar \LL_{\e}$ is given by \eqref{l2-2-2a}.
By assumptions,  $\bar \LL \bar u(x)=h(x)$ for every $x\in D$, and $\bar u\in C_c^2(D)$. Set
$\bar u(x)=0$ for every $x\in D^c$, and then extend $h$ to $\R^d$ by setting
$h(x):=\bar \LL \bar u(x)$ for all $x\in \R^d$. Then, it is not difficult to verify that
$h\in C_b^1(\R^d)$, and there is a constant $c_4>0$ such that for all $x\in \R^d$,
\begin{align}\label{t2-2-8}
|h(x)|+|\nabla h(x)|\le c_4(1+|x|)^{-d-\alpha}.
\end{align}
Then, as in the proof of Theorem \ref{t2-1} (here we also use the fact that $\bar \LL \bar u(x)=h(x)$), define for any $x\in D$,
\begin{align*}
v_\e(x):=
\begin{cases}
\bar u(x)+\e\left\langle \phi\left(x/\e\right),\nabla\bar u(x)\right\rangle
+\left(\e^\alpha\bar K^{-1}\psi\left(x/\e\right)h(x)-w_\e(x)\right)& \quad \hbox{ if } \alpha\in (1,2),\\
\bar u(x)+\e\left\langle \phi_\e\left(x/\e\right),\nabla\bar u(x)\right\rangle
+ \left(\e^\alpha\bar K^{-1}\psi\left(x/\e\right)h(x)-w_\e(x)\right)& \quad \hbox{ if } \alpha\in (0,1].
\end{cases}
\end{align*}
Here, $\phi$, $\phi_\e$ and $\psi$ are those given \eqref{t2-1-3}, and $w_\e: \R^d \to \R$
is the unique solution to the following Dirichlet exterior condition
\begin{equation}\label{t2-2-6}
\begin{cases}
\LL_\e w_\e(x)=0,  &\quad x\in D,\\
w_\e(x)=\e^\alpha\bar K^{-1}\psi\left(x/\e\right)h(x), &\quad x\in D^c.
\end{cases}
\end{equation}
The solution to the  Dirichlet problem above is given by
\begin{equation}\label{e:3.22}
w_\e(x)=\e^\alpha\bar K^{-1}\Ee_x\left[\psi\left(\e^{-1} X_{\tau_D^\e}^\e
\right)h\left(X_{\tau_D^\e}^\e\right)\right],\quad x\in D
\end{equation}
where $X^\eps:=\{X^\e_{t}\}_{t\ge 0}$ is the Hunt process associated with $\LL_\e$, and $\tau_D^\e$ is the first exit time from $D$ for the process $X^\eps$.
Indeed,  denote by $\tilde \w_\e$ the function given by the right hand side of \eqref{e:3.22}.
By  \cite{Ch}, $\tilde \w_\e$  is locally in
$W^{\alpha/2, 2}(D)$ and $\sL_\e$-harmonic in $D$.
Since $D$ is Lipschitz and so it satisfies an exterior cone condition, it follows from the heat kernel estimates in \cite{CK03} for the stable-like process
$X^\eps$ that
every boundary point is regular with respect to
the process $X^\eps$; see, e.g., \cite[Lemma 3.2]{CP} for details.   Consequently,
$\tilde \w_\e$ is continuous up to $\partial D$ and takes values $\e^\alpha\bar K^{-1}\psi\left(x/\e\right)h(x)$ on $D^c$.
 Thus, $\tilde \w_\e$ solves \eqref{t2-2-6}.
By the same argument as that for \cite[Theorem 3.8]{CP},
it can be shown that the weak solution to \eqref{t2-2-6} is unique.
Consequently $\tilde \w_\e=\w_\e$ and so \eqref{e:3.22} holds.
It follows from \eqref{e:3.22} that
there is a constant $c_5>0$ so that for all $\varepsilon\in (0,1)$,
\begin{align}\label{t2-2-6a}
\sup_{x\in D}|w_\e(x)|\le c_5\e^\alpha.
\end{align}

Now,
we split the proof into three cases.

{\bf Case 1:  $\alpha\in (1,2)$.}\,\, Using \eqref{t2-2-2}--\eqref{t2-2-4} and the fact $h\in C_b^1(\R^d)$, and repeating the proof of Theorem \ref{t2-1}, we
can prove that
$$
\LL_\e\left(\bar u(\cdot)+\e\left\langle \phi\left(\eps^{-1} \cdot \right),\nabla\bar u(\cdot)\right\rangle\right)(x)=\bar K^{-1}\bar K\left(x/\e\right)h(x)+I_1^\e(x)+\e\left\langle \phi\left(x/\e\right)
, \LL_\e\left(\nabla \bar u\right)(x)\right\rangle,
$$
\begin{align*}
\LL_\e \left(\e^\alpha\bar K^{-1}\psi\left(\eps^{-1} \cdot \right)h(\cdot)\right)(x)
=\left(1-\bar K^{-1}\bar K\left(x/\e\right)\right)h(x)+I_2^\e(x)+
\e^\alpha\bar K^{-1}\psi\left(x/\e\right)\LL_\e h(x),
\end{align*}
where
$$
\int_D ( |I_1^\e(x)|^2+|I_2^\e(x)|^2) \,dx \le c_6\e^{4-2\alpha}.
$$
Here we remark that the constant $c_6$ above only depends on $\|\nabla^k \bar u\|_\infty$ with $k=0,1,2$,  $\|h\|_\infty$ and $\|\nabla h\|_\infty$.
This, along with the equation \eqref{t2-1-3} and \eqref{t2-2-6}, yields that for every $g\in
{\rm Dom}  (\LL_\e^D)$
\begin{equation}\label{t2-2-7}
\begin{split}
\int_D\LL_\e v_\e(x)g(x)\,dx=&\int_D h(x)g(x)\,dx+\int_D (I_1^\e+I_2^\e(x))g(x)\,dx\\
&+\e\int_D \left\langle \LL_\e \left(\nabla \bar u\right)(x)\,dx,  \phi\left(x/\e\right)g(x)\right\rangle\, dx+\e^\alpha\bar K^{-1}\int_D \LL_\e h(x) \psi\left(x/\e\right)g(x)\,dx.
\end{split}
\end{equation}
Then, by the symmetry of $\LL_\e$, it holds that
\begin{align*}\label{t2-2-7a}
&\left|\e\int_D \left\langle \LL_\e \left(\nabla \bar u\right)(x),  \phi\left(x/\e\right)g(x)\right\rangle dx\right|\\
&=
\frac{\e}{2}\left|\int_{\R^d}\int_{\R^d}\left\langle \nabla \bar u(x+z)-\nabla \bar u(x), \phi\left((x+z)/{\e}\right)g(x+z)-
\phi\left(x/\e\right)g(x)\right\rangle\frac{K\left(x/\e,(x+z)/{\e}\right)}{|z|^{d+\alpha}}\,dz\,dx\right|\\
&=\frac{\e^{1-\alpha}}{2}\left|\int_{\R^d}\int_{\R^d}\left\langle \nabla \bar u(x+\e z)-\nabla \bar u(x), \phi\left(x/\e+z\right)g(x+\e z)-
\phi\left(x/\e\right)g(x)\right\rangle\frac{K\left(x/\e,x/\e+z\right)}{|z|^{d+\alpha}}\,dz\,dx\right|\\
&\le c_7\e^{2-\alpha}\Bigg(\int_D \bigg(\int_{\R^d}\|\nabla^2 \bar u\|_\infty |z|\cdot
\left(\|\nabla \phi\|_\infty |z|\I_{\{|z|\le 1\}}+\|\phi\|_\infty\I_{\{|z|>1\}}\right)\frac{1}{|z|^{d+\alpha}}\,dz\bigg)\cdot |g(x)|\,dx\Bigg)\\
&\quad+c_7\e^{1-\alpha}\Bigg(\int_{\R^d} \|\phi\|_\infty\cdot \bigg(\int_{\R^d}\left|\nabla \bar u(x+\e z)-\nabla \bar u(x)\right|
\left|g(x+\e z)-g(x)\right|\frac{1}{|z|^{d+\alpha}}\,dz\bigg)\,dx\Bigg)\\
&\le c_8\e^{2-\alpha}\int_D|g(x)|\, dx+ c_8\e \int_{\R^d}\int_{\R^d}
\left|\nabla \bar u(x+z)-\nabla \bar u(x)\right|
\left|g(x+z)-g(x)\right|\frac{1}{|z|^{d+\alpha}}\,dz\,dx\\
&\le \frac{\sE^\e(g,g)}{8}+c_9\left(\e^{4-2\alpha}+\e^2\int_{\R^d}\int_{\R^d}
\left|\nabla \bar u(x+ z)-\nabla \bar u(x)\right|^2\frac{1}{|z|^{d+\alpha}}\,dz\,dx\right)\\
&\le
\frac{\sE^\e(g,g)}{8}+c_{10}\e^{4-2\alpha},
\end{align*}
where in the third
inequality we used \eqref{l2-9-1}, the Young inequality and
$$\sE^\e(g,g):=\frac{1}{2}\iint_{\R^d\times\R^d}\left(g(x)-g(y)\right)^2\frac{K\left(x/\e,y/\e\right)}{|x-y|^{d+\alpha}}\,dx\,dy.$$

Using \eqref{t2-2-8} and following the arguments above, we can obtain
\begin{align*}
& \left|\e^\alpha\int_D \LL_\e h(x)\psi\left(x/\e\right)g(x)\,dx\right|
\le \frac{\sE^\e(g,g)}{8}+c_{11}\e^{4-2\alpha}.
\end{align*}

Putting
all the estimates above together and taking $g=u_\e-v_\e$, we have
\begin{align*}
 \sE^\e(u_\e-v_\e,u_\e-v_\e)&=-\int_D \LL_\e(u_\e-v_\e)(x)\cdot (u_\e(x)-v_\e(x))\,dx\\
&\le \frac{\sE^\e(u_\e-v_\e,u_\e-v_\e)}{4}+\int_D |u_\e(x)-v_\e(x)|(|I_1^\e(x)|+|I_2^\e(x)|)\,dx+c_{12}\e^{4-2\alpha}\\
&\le \frac{\sE^\e(u_\e-v_\e,u_\e-v_\e)}{2}+c_{13}\e^{4-2\alpha},
\end{align*}
where in the last inequality we
also used \eqref{l2-9-1} and the Young inequality.

This along with \eqref{l2-9-1} again (also by noting that $v_\e\in
W_0^{\alpha/2,2}(D)$ due to the definition of $w_\e$ on $D^c$) gives us that
\begin{align*}
\int_D|u_\e(x)-v_\e(x)|^2\,dx\le c_{14}\sE^\e(u_\e-v_\e,u_\e-v_\e)\le c_{15}\e^{4-2\alpha}.
\end{align*}

On the other hand, according to \eqref{t2-2-6a} and the definition of $v_\e$, it holds that
\begin{align*}
\int_D|\bar u(x)-v_\e(x)|^2dx\le c_{16}\e^2.
\end{align*}
Therefore, we
get
\begin{align*}
\int_D|u_\e(x)-\bar u(x)|^2\,dx\le 2\left(\int_D|u_\e(x)-v_\e(x)|^2\,dx+\int_D|\bar u(x)-v_\e(x)|^2\,dx\right)
\le c_{17}\e^{4-2\alpha}.
\end{align*}

{\bf Case 2: $\alpha=1$.}\,\, Repeating the arguments for the case $\alpha\in (1,2)$ and using \eqref{t2-1-4} to replace
$\|\nabla \phi\|_\infty$, we can show that
\begin{align*}
\int_D|u_\e(x)-\bar u(x)|^2\,dx\le c_{18}\e^2(1+|\log \e|)^4.
\end{align*} The details are omitted here.

{\bf Case 3: $\alpha\in (0,1)$.}\,\, Using  \eqref{t2-2-2}--\eqref{t2-2-5} and applying the arguments for the
case $\alpha\in (1,2)$ (here we also used the approach in the  proof of Theorem \ref{t2-1}),  we
can prove that for every $g\in
{\rm Dom}  (\LL_\e^D)$,
\begin{align*}
\int_D\LL_\e v_\e(x)g(x)\,dx=&\int_D h(x)g(x)\,dx+\int_D (J_1^\e+J_2^\e(x))g(x)\,ds+\e\int_D \left\langle \LL_\e \left(\nabla \bar u\right)(x)\,dx,  \phi_\e\left(x/\e\right)g(x)\right\rangle \,dx\\
&+\e^\alpha\bar K^{-1}\int_D \LL_\e h(x) \psi\left(x/\e\right)g(x)\,dx,
\end{align*}
where
$$
\int_D (|J_1^\e(x)|^2+|J_2^\e(x)|^2) \,dx \le c_{19}\e^{2\alpha}.
$$
Then,
\begin{align*}
&\left|\e\int_D \left\langle \LL_\e \left(\nabla \bar u\right)(x),  \phi_\e\left(x/\e\right)g(x)\right\rangle \,dx\right|\\
&=
\frac{\e}{2}\left|\int_{\R^d}\int_{\R^d}\left\langle \nabla \bar u(x+z)-\nabla \bar u(x), \phi_\e\left((x+z)/{\e}\right)g(x+z)-
\phi_\e\left(x/\e\right)g(x)\right\rangle\frac{K\left(x/\e,(x+z)/{\e}\right)}{|z|^{d+\alpha}}\,dz\,dx\right|\\
&\le c_{20}\e\Bigg(\int_D \bigg(\int_{\R^d} \|\phi_\e\|_\infty\cdot
\left(\|\nabla \bar u\|_\infty  |z|\I_{\{|z|\le 1\}}+\|\bar u\|_\infty\I_{\{|z|>1\}}\right)\frac{1}{|z|^{d+\alpha}}\,dz\bigg)\cdot |g(x)|\,dx\Bigg)\\
&\quad +c_{20}\e\Bigg(\int_{\R^d} \|\phi_\e\|_\infty\cdot \bigg(\int_{\R^d}\left|\nabla \bar u(x+z)-\nabla \bar u(x)\right|
\left|g(x+z)-g(x)\right|\frac{1}{|z|^{d+\alpha}}\,dz\bigg)\,dx\Bigg)\\
&\le c_{21}\e^{\alpha}\int_D|g(x)|\, dx+ c_{21}\e^\alpha \int_{\R^d}\int_{\R^d}
\left|\nabla \bar u(x+z)-\nabla \bar u(x)\right|
\left|g(x+z)-g(x)\right|\frac{1}{|z|^{d+\alpha}}\,dz\,dx\\
&\le \frac{\sE^\e(g,g)}{8}+c_{22}\left(\e^{2\alpha}+\e^{2\alpha}\int_{\R^d}\int_{\R^d}
\left|\nabla \bar u(x+ z)-\nabla \bar u(x)\right|^2\frac{1}{|z|^{d+\alpha}}\,dz\,dx\right)\\
&\le
\frac{\sE^\e(g,g)}{8}+c_{23}\e^{2\alpha}.
\end{align*}

Similarly, we can show that
\begin{align*}
&\left|\e^\alpha\int_D \LL_\e h(x)\psi\left(x/\e\right)g(x)\,dx\right|
\le \frac{\sE^\e(g,g)}{8}+c_{24}\e^{2\alpha}.
\end{align*}
Putting all the estimates above together and taking $g=u_\e-v_\e$, we have
\begin{align*}
\sE^\e(u_\e-v_\e,u_\e-v_\e)&=-\int_D \LL_\e(u_\e-v_\e)(x)\cdot (u_\e(x)-v_\e(x))\,dx\\
&\le \frac{\sE^\e(u_\e-v_\e,u_\e-v_\e)}{4}+\int_D |u_\e(x)-v_\e(x)|(|J_1^\e(x)|+|J_2^\e(x)|)\,dx+c_{25}\e^{2\alpha}\\
&\le \frac{\sE^\e(u_\e-v_\e,u_\e-v_\e)}{2}+c_{26}\e^{2\alpha},
\end{align*}
where we used  \eqref{l2-9-1} and the Young inequality in the last inequality.
This, along with \eqref{l2-9-1} again, yields that
\begin{align*}
\int_D|u_\e(x)-v_\e(x)|^2\,dx\le c_{27}\sE^\e(u_\e-v_\e,u_\e-v_\e)\le c_{28}\e^{2\alpha}.
\end{align*}
Hence, it holds that
\begin{align*}
\int_D|u_\e(x)-\bar u(x)|^2\,dx\le 2\left(\int_D|u_\e(x)-v_\e(x)|^2\,dx+\int_D|\bar u(x)-v_\e(x)|^2\,dx\right)
\le c_{29}\e^{2\alpha}.
\end{align*}
Now we have finished the proof.
\end{proof}

\bigskip

\noindent {\bf Acknowledgements.}\,\,
The research of Xin Chen is supported by the National Natural Science Foundation
of China (No. 12122111).
The research of Zhen-Qing Chen is partially supported by Simons Foundation grant 520542.
 The research of Takashi Kumagai is supported
by JSPS KAKENHI Grant Number JP22H00099 and 23KK0050.
The research of Jian Wang is supported by the National Key R\&D Program of China (2022YFA1006003) and  the National Natural Science Foundation of China (Nos.\ 12071076 and 12225104).

\vskip 0.3truein
{\small
{\bf Xin Chen:}
   School of Mathematical Sciences, Shanghai Jiao Tong University, 200240 Shanghai, P.R. China. \newline \texttt{chenxin217@sjtu.edu.cn}
	
\bigskip
	
{\bf Zhen-Qing Chen:}
   Department of Mathematics, University of Washington, Seattle,
WA 98195, USA. \newline \texttt{zqchen@uw.edu}

\bigskip

{\bf Takashi Kumagai:}
 Department of Mathematics,
Waseda University, Shinjuku, Tokyo, 169-8555, Japan.
\newline \texttt{t-kumagai@waseda.jp}

\bigskip

{\bf Jian Wang:}
    School of Mathematics and Statistics \& Key Laboratory of Analytical Mathematics and Applications (Ministry of Education) \& Fujian Provincial Key Laboratory
of Statistics and Artificial Intelligence, Fujian Normal University, 350007 Fuzhou, P.R. China.
     \newline \texttt{jianwang@fjnu.edu.cn}

     }

\end{document}